\documentclass[a4paper, 12pt]{amsart}
\usepackage{hyperref, amssymb}
\usepackage[all]{xy}
\usepackage{enumerate}
\usepackage{xspace}
\usepackage[table]{xcolor}
\usepackage{multirow}
\usepackage{tikz}
\usepackage[margin]{fixme} 
\newcommand{\thecase}[1]{\noindent\textsc{Case #1:}\ }

\newcommand{\setof}[2]{\left\{ #1 \;\middle|\; #2 \right\}}
\newcommand{\n}{n}

\newtheorem{theor}{Theorem}[section]
\newtheorem{propo}[theor]{Proposition}
\newtheorem{lemma}[theor]{Lemma}
\newtheorem{corol}[theor]{Corollary}

\theoremstyle{defin}
\newtheorem{defin}[theor]{Definition}

\theoremstyle{remark}
\newtheorem{remar}[theor]{Remark}

\numberwithin{equation}{section}

\newcommand{\nn}{\underline{n}}
\newcommand{\mm}{\underline{m}}

\newcommand{\K}{\mathbb{K}}
\newcommand{\CC}{\mathbb{C}}
\newcommand{\NN}{\mathbb{N}}

\newcommand{\ZZ}{\mathbb{Z}}

\newcommand{\IIIm}{\mbox{\texttt{\textup{(I-)}}}}
\newcommand{\IIIp}{\mbox{\texttt{\textup{(I+)}}}}

\newcommand{\OOO}{\mbox{\texttt{\textup{(O)}}}}
\newcommand{\CCC}{\mbox{\texttt{\textup{(C)}}}}
\newcommand{\CCCp}{\mbox{\texttt{\textup{(C+)}}}}
\newcommand{\RRR}{\mbox{\texttt{\textup{(R)}}}}
\newcommand{\RRRp}{\mbox{\texttt{\textup{(R+)}}}}
\newcommand{\PPP}{\mbox{\texttt{\textup{(P)}}}}
\newcommand{\PPPp}{\mbox{\texttt{\textup{(P+)}}}}

\newcommand{\calP}{\ensuremath{\mathcal{P}}\xspace}

\newcommand{\MPZccc}[1][\mathbf{m}\times\mathbf{n}]{\ensuremath{\mathfrak{M}^{\circ\circ\circ}_\calP(#1,\ZZ)}\xspace}

\newcommand{\MPZc}[1][\mathbf{m}\times\mathbf{n}]{\ensuremath{\mathfrak{M}^\circ_\calP(#1,\ZZ)}\xspace}
\newcommand{\MPZcc}[1][\mathbf{m}\times\mathbf{n}]{\ensuremath{\mathfrak{M}^{\circ\circ}_\calP(#1,\ZZ)}\xspace}

\newcommand{\SL}{\operatorname{SL}_\Gamma}
\newcommand{\GL}{\operatorname{GL}_\Gamma}
\newcommand{\SLp}{\operatorname{SL}^+_\Gamma}
\newcommand{\GLp}{\operatorname{GL}^+_\Gamma}
\newcommand{\MG}[1][]{\mathcal M_{\Gamma_{#1}}}
\renewcommand{\AA}{\mathsf A}
\newcommand{\OO}{\underline 0}
\newcommand{\EE}{\underline 1}
\newcommand{\BB}{\mathsf B}
\newcommand{\XX}{\mathsf X}
\renewcommand{\CC}{\mathsf C}\newcommand{\DD}{\mathsf D}
\newcommand{\DBpair}{(\DD,\BB)}
\newcommand{\DBpairprime}{(\DD',\BB')}
\newcommand{\DBpairprimeprime}{(\DD'',\BB'')}
\newcommand{\cok}{\operatorname{cok}}
\newcommand{\im}{\operatorname{im}}
\newcommand{\myop}{\mathsf{E}_{ij}}

\newcommand{\mypair}[2]{\left(#1,#2\right)}
\definecolor{lgray}{rgb}{0.9,0.9,0.9}

\newcommand{\vt}[1]{\mbox{\fbox{$#1$}}}

\date{\today}

\begin{document}
\begin{abstract}
We geometrically describe the relation induced on a set of graphs by isomorphism of their associated graph $C^*$-algebras as the smallest equivalence relation generated by five types of moves.
The graphs studied have
finitely many vertices and finitely or countably infinitely many edges, corresponding to unital and separable $C^*$-algebras.
\end{abstract}

	\author{Sara E.\ Arklint}
        \address{Department of Mathematical Sciences \\
        University of Copenhagen\\
        Universitetsparken~5 \\
        DK-2100 Copenhagen, Denmark}
        \email{arklint@math.ku.dk }
	
		\author{S{\o}ren Eilers}
        \address{Department of Mathematical Sciences \\
        University of Copenhagen\\
        Universitetsparken~5 \\
        DK-2100 Copenhagen, Denmark}
        \email{eilers@math.ku.dk }

	\author{Efren Ruiz}
        \address{Department of Mathematics\\University of Hawaii,
Hilo\\200 W. Kawili St.\\
Hilo, Hawaii\\
96720-4091 USA}
        \email{ruize@hawaii.edu}
\date{\today}

\title{Geometric classification of isomorphism of unital graph $C^*$-algebras}
\maketitle

\newcommand{\companion}[1]{\cite[\ref{#1}]{seer:rmsigc}}
\newcommand{\companionpaper}{\cite{seer:rmsigc}}
\section{Introduction}

As is often the case, the classification problem for $*$-isomorphism and for stable isomorphism among unital graph $C^*$-algebras were completed in tandem, in this case because the authors of \cite{segrerapws:ccuggs} were able to extract the exact result as a corollary to the stabilized one. A key element in the proof of the stabilized result --- highly interesting in itself --- was the geometric description of the equivalence relation among such graphs induced by having the same graph $C^*$-algebra up to stable isomorphism, which in \cite{segrerapws:ccuggs} was given a concrete description as the smallest equivalence relation containing a list of \emph{moves} changing one graph into another in an invariant fashion. In other words, it was established that  two such graphs define the same stabilized graph $C^*$-algebras if and only it is possible to transform one into the other by a finite number of moves of this type (or their inverses), much in the way that Reidemeister moves determine homotopy of knots.

The work in \cite{segrerapws:ccuggs} did not provide an answer to the natural question of whether or not such a geometric description could be obtained for exact $*$-isomorphism, and we present a positive solution to that question here. Using variations of the moves used in  \cite{segrerapws:ccuggs}  which have been carefully chosen to preserve the $C^*$-algebra itself rather than its stabilization whilst retaining the necessary flexibility, we show that  the equivalence relation among such graphs induced by having the same graph $C^*$-algebra ``on the nose'' may also be given a concrete description as the smallest equivalence relation containing a list of geometric moves.


Our strategy of proof is by now standard, an elaboration of the original approach by Franks \cite{jf:fesft} to classify irreducible shifts of finite type up to flow equivalence which draws significantly on previous refinements by Boyle and Huang (\cite{mbdh:pbeim},\cite{mb:fesftpf}) and by two of the authors with Restorff and S\o{}rensen \cite{segrerapws:ccuggs}. The key idea is to transform the question into one in algebra by proving that fundamental matrix operations such as row and column addition to adjacency matrices defined by the graphs under study are generated by a succession of legal moves.   Our tool for doing this is an \emph{antenna calculus} developed for this purpose which is used to represent the information remembered by the exact $*$-isomorphism class (but forgotten after stabilization) by means of simple auxiliary configurations in such graphs, or -- equivalently, as we shall see --  as a vector complementing the adjacency matrix.

The work presented here was initiated during the 2016 Institut Mittag-Leffler focus program ``Classification of operator algebras: complexity, rigidity, and dynamics'' where we proved that $*$-isomorphism among unital graph $C^*$-algebras was generated by a list of specific moves, and we presented our results at the  program's workshop ``Classification
and discrete structures''. This work contained the development of the \RRRp\ move which is essential here, but one of the other moves on our list was so arithmetic in nature that we couldn't really defend calling our result geometric, and hence we refrained from publishing the result.

Recently  the second and third authors have initiated in \cite{seer:rmsigc} a systematic study of a refined collection of moves with the property that relevant subclasses of these moves generate isomorphisms of the graph $C^*$-algebras which respect additional structure such as diagonals and the canonical gauge action, and we were led to the definition of a refined type of insplitting -- the \IIIp\ move -- which not only induces $*$-isomorphism as opposed to the original version's stable isomorphism, but also respects the additional structure mentioned above. To our immense satisfaction we have been able to show that the arithmetic move abandoned by the authors can be induced by the \IIIp\ move along with the other honestly geometric moves already on our list, and hence we are now able to present a list of natural moves which all induce $*$-isomorphism, and prove that any $*$-isomorphism is induced by these moves by appealing to the argument we developed several years ago.

The bulk of the paper is devoted to the first step of Franks' approach: To show that the elementary matrix operations may be implemented by geometric moves. This is particularly tricky in the situation studied here, but the most challenging technical difficulty is the same in all such problems: To ensure that the matrices visited as one tries to implement the given data are non-negative in an appropriate sense allowing them to make sense as adjacency matrices for intermediate graphs as well. For this, thankfully, we may appeal to the complicated analysis in \cite{segrerapws:ccuggs} with rather minor adjustments.

In the interest of brevity, we relegate the proof that the improved moves indeed respect the exact isomorphism class of the $C^*$-algebras to the companion paper \cite{seer:rmsigc}, but we will describe them with some care below.

\subsection{Acknowledgements}

The first and second named authors were supported by 
the Danish National Research Foundation through the Centre for Symmetry and Deformation
(DNRF92).
The second named author was further supported by the DFF-Research
Project 2 `Automorphisms and Invariants of Operator Algebras', no.\ 7014-00145B. The third named author was supported by
a Simons Foundation Collaboration Grant, \# 567380.

During the initial phases of this work the first named author was a postdoctoral fellow at the Mittag-Leffler Institute. All authors thank the Institute and its staff for the excellent working conditions provided.

\section{Preliminaries}

\subsection{Notation and conventions}

We use the definition of graph $C^*$-algebras in \cite{njfmlir:cig} in which sinks and infinite emitters are singular vertices, and always consider $C^*(E)$ as a universal $C^*$-algebra generated by Cuntz-Krieger families $\{s_e,p_v\}$ with $e$ ranging over edges and $v$ ranging over vertices in $E$.

Unless stated otherwise, graphs $E$, $F$ will always be considered as having finitely many vertices and finitely or countably infinitely many edges. We generally follow notation from \cite{segrerapws:ccuggs,segrerapws:gcgcfg}, but will deviate slightly from these papers when describing graphs by matrices as explained below.

\subsection{Legal moves}

We briefly introduce the five types of moves we are considering. See \companionpaper\ for a full discussion.

In- and out-splitting works as in symbolic dynamics by distributing the incoming (resp. outgoing) edges to new vertices according to a given partition, and duplicating the outgoing (resp. incoming) ones. Note, however, the lack of symmetry below: Out-splitting may take place everywhere, but the partition cannot contain empty sets. In-splitting is restricted to regular vertices, but empty sets are allowed.

\begin{defin}[Move \OOO: Outsplit at a non-sink]\label{OOO}
Let $E = (E^0 , E^1 , r, s)$ be a graph, and let $w \in E^0$ be a vertex that is not a sink. 
Partition $s^{-1} (w)$ as a disjoint union of a finite number of nonempty sets
$$s^{-1}(w) = \mathcal{E}_1\sqcup \mathcal{E}_2\sqcup \cdots \sqcup\mathcal{E}_n$$
with the property that at most one of the $\mathcal{E}_i$ is infinite. 
Let $E_O$ denote the graph $(E_O^0, E_O^1, r_O , s_O )$ defined by 
\begin{align*}
E_O^0&:= \setof{v^1}{v \in E^0\text{ and }v \neq w}\cup\{w^1, \ldots, w^n\} \\
E_O^1&:= \setof{e^1}{e \in E^1\text{ and }r(e) \neq w}\cup \setof{e^1, \ldots , e^n}{e \in E^1\text{ and }r(e) = w} \\
r_{E_O} (e^i ) &:= 
\begin{cases}
r(e)^1 & \text{if }e \in E^1\text{ and }r(e) \neq w\\
w^i & \text{if }e \in E^1\text{ and }r(e) = w
\end{cases} \\
s_{E_O} (e^i ) &:= 
\begin{cases}
s(e)^1 & \text{if }e \in E^1\text{ and }s(e) \neq w \\
s(e)^j & \text{if }e \in E^1\text{ and }s(e) = w\text{ with }e \in \mathcal{E}_j.
\end{cases}
\end{align*}
We  say $E_O$ is formed by
performing move \OOO\ to $E$.
\end{defin}

\begin{defin}[Move \IIIm:  Insplitting] \label{IIIm}
Let $E = ( E^0, E^1, r_E , s_E )$ be a graph and let $w \in E^0$ be a regular vertex.  Partition $r^{-1} (w)$ as a finite disjoint union of (possibly empty) subsets,
\[
r^{-1} (w) = \mathcal{E}_1 \sqcup \mathcal{E}_2 \sqcup \cdots \sqcup \mathcal{E}_n.
\]

Let $E_I = ( E_I^0 , E_I^1, r_{E_I} , s_{E_I} )$ be the graph defined by 
\begin{align*}
E^0_I &= \{ v^1 \ : \ v \in E^0 \setminus \{ w \} \} \cup \{ w^1, w^2, \ldots, w^n \} \\
E^1_I &= \{ e^1 \ : \ e \in E^1, s_E (e) \neq w \} \cup \{ e^1 , e^2 , \ldots, e^n \ : \ e \in E , s_E (e)  = w \} \\
s_{E_I}(e^i ) &= \begin{cases} s_E(e)^1 &\text{if $e \in E^1, s_E (e) \neq w$} \\ w^i &\text{if $e \in E^1, s_E (e) = w$} \end{cases} \\
r_{E_I}(e^i ) &= \begin{cases} r_E(e)^1 &\text{if $e \in E^1, r_E (e) \neq w$} \\ w^j &\text{if $e \in E^1, r_E (e) = w, e \in \mathcal{E}_j$} \end{cases}
\end{align*}
We  say $E_I$ is formed by
performing move \IIIm\ to $E$.
\end{defin}

\begin{defin}[Move \IIIp:  Unital Insplitting] \label{IIIp}
The graphs $E$ and $F$ are said to be \emph{move $\mbox{\texttt{\textup{(I+)}}}$ equivalent} if there exists a graph $G$ and a regular vertex $w\in G^0$ such that $E$ and $F$ are both the result of an \IIIm\ move applied to  $G$ via a partition of $r_{G}^{-1}(w)$ using $n$ sets.
\end{defin}

Note that we do \textbf{not} consider the \IIIm\ move further in this paper --- it leaves  $C^*(E)\otimes\K$ invariant, but not $C^*(E)$. It is convenient to think of an $\IIIp$ move as the result of redistributing the past of vertices having the same future. For instance we have
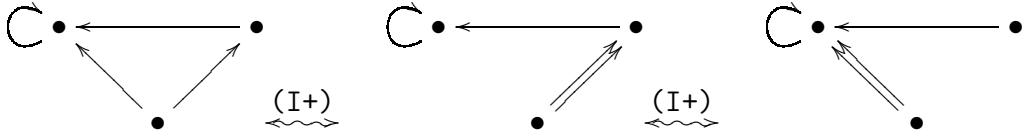
\[
\xymatrix{\bullet\ar@(dl,ul)[]&&\bullet\ar[ll]&&\bullet\ar@(dl,ul)[]&&\bullet\ar[ll]&&\bullet\ar@(dl,ul)[]&&\bullet\ar[ll]\\
&\bullet\ar[ur]\ar[ul]&\ar@{<~>}[r]^{\IIIp}&&&\bullet\ar@<-0.5mm>[ur]\ar@<0.5mm>[ur]&\ar@{<~>}[r]^{\IIIp}&&&\bullet\ar@<-0.5mm>[ul]\ar@<0.5mm>[ul]}
\]
since all graphs may be obtained by an \IIIm\ move applied to $\quad\quad\xymatrix{\bullet\ar@(dl,ul)[]&\bullet\ar@<-0.5mm>[l]\ar@<0.5mm>[l]}$ with two sets in the partition.

\begin{defin}[Move $\RRRp$:  Unital Reduction]\label{RRRp}
Let $E$ be a graph and let $w$ be a regular vertex which does not support a loop.  Let $E_{R+}$ be the graph defined by 
\begin{align*}
E_{R+}^0 &= ( E^0  \setminus \{ w \} ) \sqcup \{ \widetilde{w} \} \\
E_{R+}^1 &= \left( E^1 \setminus (  r_E^{-1} (w) \cup s_E^{-1} (w)  ) \right) \sqcup \{ [ ef ] \ : \ e \in r_E^{-1}(w) , f \in s_{E}^{-1}(w) \} \sqcup \{ \widetilde{f} \ : \ f\in s_E^{-1}(w) \}
\end{align*}
where the source and range maps of $E_{R+}$ extend those of $E$, and satisfy $s_{E_{R+}} ( [ef] ) = s_E(e)$, $s_{E_{R+}}( \widetilde{f} ) = \widetilde{w}$, $r_{E_{R+}} ( [ef] ) = r_{E}(f)$ and $r_{E_{R+}} ( \widetilde{f} ) = r_E(f)$.
\end{defin}

The \RRRp\ move is best thought of as the result of removing a vertex and replacing all two-step paths through it by direct paths. The outgoing edges from the deleted vertex are preserved as edges from a source.

The \CCCp\ and \PPPp\ moves are defined by gluing on small graphs to the existing one under very precisely given conditions. Examples are
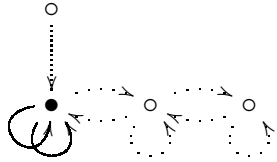
\[
\xymatrix{\circ\ar@{..>}[d]&\\\bullet\ar@(l,d)[]\ar@(dl,dr)[]\ar@{..>}@/^/[r]&\circ \ar@{..>}@/^/[r]\ar@(dl,dr)@{..>}[]\ar@{..>}@/^/[l]&\circ\ar@{..>}@/^/[l]\ar@{..>}@(dl,dr)[]\\&}
\]
and
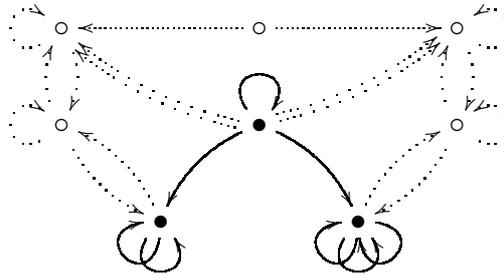
\[
\xymatrix{&\circ\ar@{..>}@(ld,lu)[]\ar@{..>}@/^/[d]&&\circ\ar@{..>}[ll]\ar@{..>}[rr]&&\circ\ar@{..>}@(rd,ru)[]\ar@{..>}@/^/[d]&\\
&\circ\ar@{..>}@(ld,lu)[]\ar@{..>}@/^/[u]\ar@{..>}@/_/[dr]&&\bullet\ar@{..>}@/_/[urr]\ar@{..>}@/_/@<1mm>[urr]\ar@{..>}@/^/[ull]\ar@{..>}@/^/@<1mm>[ull]\ar@(ul,ur)[]\ar@/^/[dr]\ar@/_/[dl]&&\circ\ar@{..>}@/^/[u]\ar@{..>}@(rd,ru)[]\ar@{..>}@/^/[dl]&\\
&&\bullet\ar@(d,l)[]\ar@(ld,rd)[]\ar@{..>}@/_/[ul]&&\bullet\ar@(d,l)[]\ar@(ld,rd)[]\ar@(r,d)[]\ar@{..>}@/^/[ur]&\\&&&}
\]
where the new parts of the graphs are indicated with unfilled vertices and dotted arrows.
Since they will not play a very central role in the arguments in the present paper, we refer to \companionpaper\ for a full discussion. 

Throughout the paper we say ``moves of type ${\mbox{\texttt{\textup{(X)}}}}$'' when we refer to a collection of such moves and their inverses.
This applies in particular to the collection of moves of type \OOO, \IIIp, and \RRRp. In fact the remaining two types are in an appropriate  sense their own inverses.

\begin{theor}[\cite{seer:rmsigc}]\label{fromcompanion}
When $F$ is obtained from $E$ by one of the moves
\[
\OOO,\IIIp,\RRRp,\CCCp,\PPPp
\]
then $C^*(E)\simeq C^*(F)$.
\end{theor}

\subsection{$\mathsf{ABCD}$-matrices}

We shall describe graphs  by adjacency matrices in two equivalent ways, one  slightly less standard than the other, which is eminently adjusted to the technical needs in the paper. The starting point for this notation is the observation that whenever a graph contains two or more regular sources, they may be collected to one by an \OOO\ move in reverse, and since we may pass freely between graphs obtained by applying moves on our list, we may always assume that  there is at most one such vertex. We enumerate the remaining vertices by $1,\dots,n$ and refer to a vertex by its number as $\vt{i}$. Placing a regular source first, the adjacency matrix has the form
\begin{equation}\label{acsetup}
\begin{bmatrix}0&c_1&\cdots&c_n\\
0&a_{11}&\cdots &a_{1n}\\
\vdots&\vdots&&\vdots\\
0&a_{n1}&\cdots &a_{nn}
\end{bmatrix}
\end{equation}
and we denote the submatrices with entries denoted $a_{ij}$ and $c_i$ by $\AA$ and $\CC$, respectively, thinking of $\CC$ as a row vector. 
In case there is no regular sources, we let $\CC$ denote the zero vector and think of this setting as representing the graph with adjacency matrix $\AA$. Letting all $c_i=0$ in \eqref{acsetup} would give a very different system, so it is essential to deviate from the generic construction here. We think of the regular source as being ``deleted'' when $\CC=\OO$.

In most cases we work instead of $\AA$ and $\CC$ with the pair $\DBpair$ with $\BB$ a matrix with the same dimensions as $\AA$ and $\DD$ a column vector with the same number of entries as $\CC$ given by
\begin{eqnarray}
b_{ij}&=&a_{ji}-\delta_{i,j} \quad(\text{Kronecker } \delta)\label{ab}\\
d_i&=&c_i+1\label{cd}
\end{eqnarray}
We will use round parentheses on $\BB$ and $\DD$ to set them aside from $\AA$ and $\CC$ given by bracketed matrices.
It is clear that this contains the same information, and it will be obvious from Theorems \ref{row} and \ref{col}, as well as from comparison with $K$-theory, why this notation is useful. For the latter, we indicate by $\BB^\bullet$ and $\BB^\circ$ the matrices obtained by collecting, respectively, the columns corresponding to regular and singular vertices in the graph described. Then we have

\begin{lemma}\label{Ktheory}
When the graph $E$ is represented by the pair $\DBpair$, we have
\[(K_0(C^*(E)),[1_{C^*(E)}])=(\cok \BB^\bullet,\DD+\im \BB^\bullet).\]
\end{lemma}
\begin{proof}
The result is standard (see e.g. \cite{mt:okgc}) when there is no regular source and hence $\DD=\EE$. 
If not, and the regular vertices are $i_1,\dots,i_k$, the $K_0$-group is given by
\[
\cok\begin{pmatrix}-1&0&\cdots&0\\
c_1&b_{1i_1}&\cdots &b_{1i_k}\\
\vdots&\vdots&&\vdots\\
c_n&b_{ni_1}&\cdots &b_{ni_k}
\end{pmatrix}
\]
with the class of the unit represented by $\EE$. This is isomorphic to the given data $(\cok \BB^\bullet,\DD+\im \BB^\bullet)$.
\end{proof}

Let \vt{1} and \vt{2} be different  vertices which are not regular sources. It will be  instructive to depict the graph as 
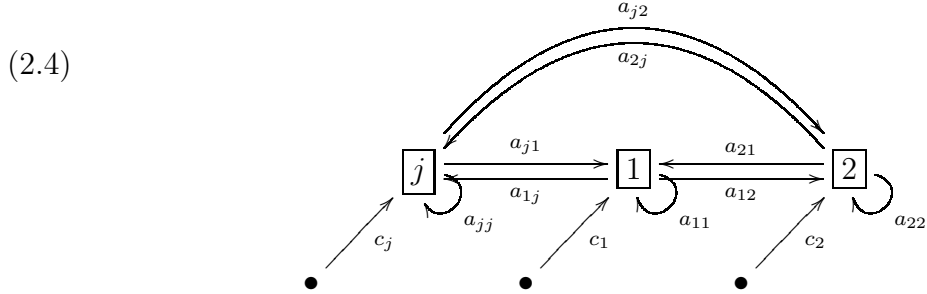
\begin{equation}
\xymatrix{
&&&&&\\
&\vt{j}\ar@/^/@(r,d)[]^{a_{jj}}\ar@<1mm>@/^1.8cm/[rrrr]^{a_{j2}}\ar@<1mm>[rr]^{a_{j1}}&&\vt{1}\ar@/^/@(r,d)[]^{a_{11}}\ar@<1mm>[ll]^{a_{1j}}\ar@<-1mm>[rr]_{a_{12}}&&\vt{2}\ar@/^/@(r,d)[]^{a_{22}}\ar@<-1mm>[ll]_{a_{21}}\ar@<1mm>@/_1.8cm/[llll]^{a_{2j}}\\
\bullet\ar[ur]_{c_j}&&\bullet\ar[ur]_{c_1}&&\bullet\ar[ur]_{c_2}
}\label{initialsetup}
\end{equation}
where $\vt{j}$ with $j>2$ is an arbitrary vertex which is not a regular source, and entries $a_{ij}$ are in the full range $0\leq a_{ij}\leq \infty$.  Note that entries $c_i$ must be finite, and we use the convention that if some $c_1,c_2$ or $c_j$ is zero, there is no corresponding source.
Note that we may pass between this description of the vertices emanating from each ``shadow source'' to the corresponding non-source and the setup in which there is only one source by a single \OOO\ move, so we will choose the one which is most visually convenient, which in most cases is the former, and we think of the edges enumerated by the $c_i$ as \emph{antennae} attached to an original graph.

\section{Elementary matrix operations}

In this section we show how elementary matrix operations are induced on the $\DBpair$ pair by our moves \OOO, \IIIp\ and \RRRp\ applied to the graphs they represent.  We follow the strategy of imposing any condition on the configuration necessary to establish the claims easily. In the ensuing sections we then proceed to remove many of these conditions.

\subsection{Outsplitting gives row operations}

In this subsection we study how outsplitting translates to row operations on the $\AA$ or $\BB$ matrices which also influence the   $\CC$ and $\DD$ vectors in a systematic way.


We start at \eqref{initialsetup}
and assume there 
 is at least one edge from \vt{1} to \vt{2}, and \vt{1} emits at least one other edge (to any other vertex or to \vt{1} or \vt{2}). Then we can outsplit at \vt{1} with one set in the partition being the selected edge from \vt{1} to \vt{2} and another containing the rest, and \OOO\ gives us\\[0.1cm]
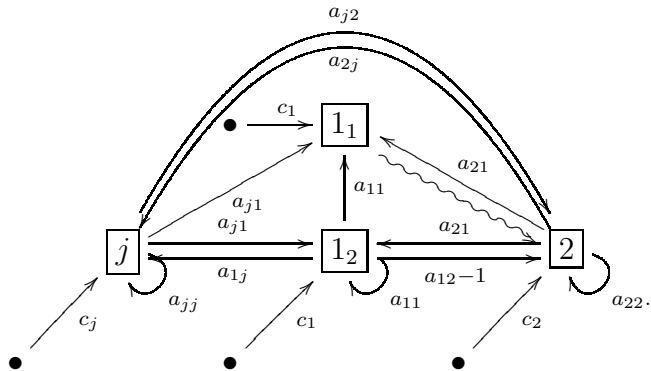
\[
\xymatrix{
&&\bullet\ar[r]^{c_1}&\vt{1_1}\ar@{~>}@<-1mm>[drr]&&\\
&\vt{j}\ar@/^/@(r,d)[]^{a_{jj}}\ar@<1mm>@/^2.8cm/[rrrr]^{a_{j2}}\ar@<1mm>[rr]^{a_{j1}}\ar[urr]_{a_{j1}}&&\vt{1_2}\ar@/^/@(r,d)[]^{a_{11}}\ar@<1mm>[ll]^{a_{1j}}\ar[u]_{a_{11}}\ar@<-1mm>[rr]_{a_{12}-1}&&\vt{2}\ar@/^/@(r,d)[]^{a_{22}.}\ar@<-1mm>[ll]_{a_{21}}\ar@<-1mm>[ull]_{a_{21}}\ar@<1mm>@/_2.8cm/[llll]^{a_{2j}}\\
\bullet\ar[ur]_{c_j}&&\bullet\ar[ur]_{c_1}&&\bullet\ar[ur]_{c_2}
}
\]
The squiggly arrow is alone in the sense that there is nothing else from \vt{1_1} to \vt{2}, and we know that \vt{1_2} is not a sink.

Now we perform \RRRp\ to \vt{1_1} which we know is regular and does not support a loop, and the situation becomes
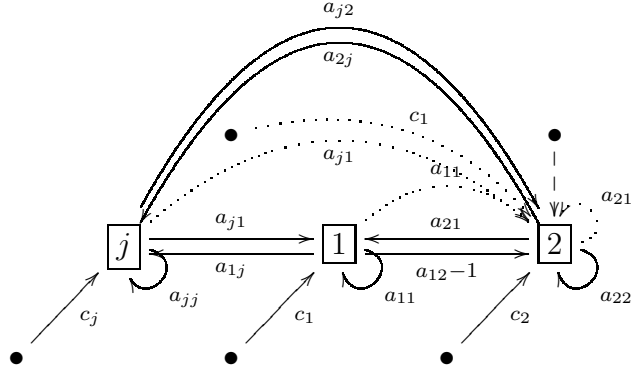
\[
\xymatrix{
&&\bullet\ar@{..>}@/^8mm/[drrr]^{c_1}&&&\bullet\ar@{-->}[d]\\
&\vt{j}\ar@/^/@(r,d)[]^{a_{jj}}\ar@<1mm>@/^2.8cm/[rrrr]^{a_{j2}}\ar@{..>}@/^14mm/[rrrr]_{a_{j1}}\ar@<1mm>[rr]^{a_{j1}}&&\vt{1}\ar@/^/@(r,d)[]^{a_{11}}\ar@<1mm>[ll]^{a_{1j}}\ar@{..>}@/^8mm/[rr]^{a_{11}}\ar@<-1mm>[rr]_{a_{12}-1}&&\vt{2}\ar@/^/@(r,d)[]^{a_{22}}\ar@{..>}@(r,u)[]_-{a_{21}}\ar@<-1mm>[ll]_{a_{21}}\ar@<1mm>@/_2.8cm/[llll]^{a_{2j}}\\
\bullet\ar[ur]_{c_j}&&\bullet\ar[ur]_{c_1}&&\bullet\ar[ur]_{c_2}
}
\]
Here all the dotted edges are induced by paths that used to go via \vt{1_1}, and the slashed edge is the extra source introduced by \RRRp. We can collect sources and  redraw this as 
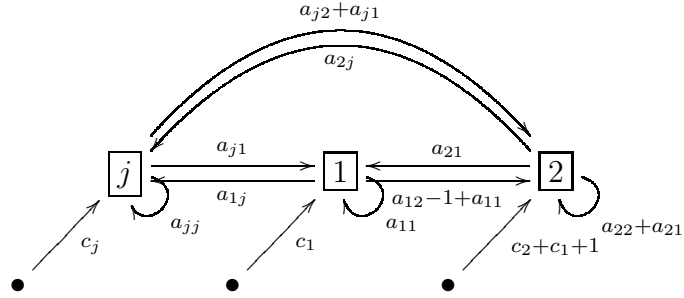
\[
\xymatrix{
&&&&&\\
&\vt{j}\ar@/^/@(r,d)[]^{a_{jj}}\ar@<1mm>@/^1.8cm/[rrrr]^{a_{j2}+a_{j1}}\ar@<1mm>[rr]^{a_{j1}}&&\vt{1}\ar@/^/@(r,d)[]^{a_{11}}\ar@<1mm>[ll]^{a_{1j}}\ar@<-1mm>[rr]_{a_{12}-1+a_{11}}&&\vt{2}\ar@/^/@(r,d)[]^{a_{22}+a_{21}}\ar@<-1mm>[ll]_{a_{21}}\ar@<1mm>@/_1.8cm/[llll]^{a_{2j}}\\
\bullet\ar[ur]_{c_j}&&\bullet\ar[ur]_{c_1}&&\bullet\ar[ur]_{c_2+c_1+1}
}
\]
and then we get:

\begin{propo}\label{row}
Given a pair $\DBpair$ describing the graph $E$. When $b_{21}>0$ and $\sum_{j=1}^n b_{j1}>0$, we can go from $E$ to the graph described by the pair $\DBpairprime$ given as
\[
\DD'=\begin{pmatrix}d_1\\d_2+d_1\\d_3\\\vdots\\d_n\end{pmatrix}\quad \BB'=\begin{pmatrix}b_{11}&b_{12}&b_{13}&\dots&b_{1\n}\\
b_{21}+b_{11}&b_{22}+b_{12}&b_{23}+b_{13}&\dots&b_{2\n}+b_{1\n}\\
b_{31}&b_{32}&b_{33}&\dots&b_{3\n}\\\vdots&\vdots&\vdots&&\vdots\\b_{n1}&b_{n2}&b_{n3}&\dots&b_{n\n}
\end{pmatrix}
\] 
by  moves of type \OOO\ and \RRRp.
\end{propo}

Recall $b_{11}$ might be negative, so the condition $\sum_{j=1}^n b_{j1}>0$ is not automatic from $b_{21}>0$.

\begin{proof}
Follow the recipe given above, and substitute by \eqref{ab} and \eqref{cd} in the conclusion.
\end{proof}


\subsection{Insplitting gives column operations}

We now pass to column operations on a given pair $\DBpair$. Let \vt{1} and \vt{2} be different vertices with \vt{1} regular. We start with the setup as in \eqref{initialsetup}
where this time we need to assert that there is at least one edge from \vt{2} to \vt{1}, and further 
that there is at least as many antennae to any \vt{j} as there are edges from  \vt{1} to \vt{j} for all $j$ (including $j\in\{1,2\}$). As above \vt{j} and \vt{2} could be singular.
The condition that $c_j\geq a_{1j}$  will allow us to redraw as
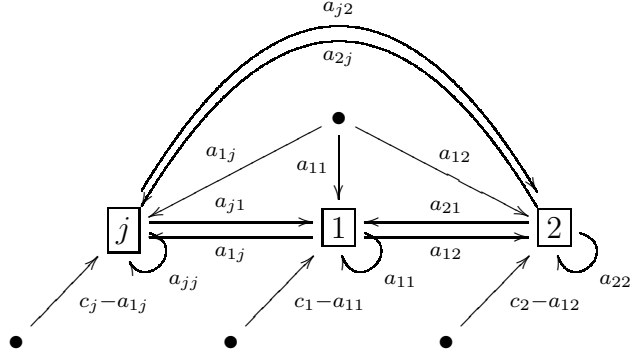
\[
\xymatrix{
&&&\bullet\ar[drr]^-{a_{12}}\ar[d]_-{a_{11}}\ar[dll]_-{a_{1j}}&&\\
&\vt{j}\ar@/^/@(r,d)[]^{a_{jj}}\ar@<1mm>@/^2.6cm/[rrrr]^{a_{j2}}\ar@<1mm>[rr]^{a_{j1}}&&\vt{1}\ar@/^/@(r,d)[]^{a_{11}}\ar@<1mm>[ll]^{a_{1j}}\ar@<-1mm>[rr]_{a_{12}}&&\vt{2}\ar@/^/@(r,d)[]^{a_{22}}\ar@<-1mm>[ll]_{a_{21}}\ar@<1mm>@/_2.6cm/[llll]^{a_{2j}}\\
\bullet\ar[ur]_{c_j-a_{1j}}&&\bullet\ar[ur]_{c_1-a_{11}}&&\bullet\ar[ur]_{c_2-a_{12}}
}
\]
(the source in the middle is never a sink, and if any number $c_j-a_{1j}$ is zero, the source is deleted). Renaming the middle vertices
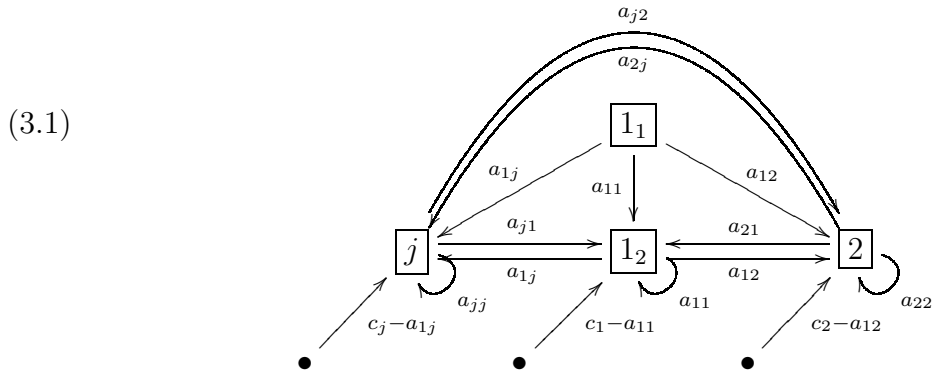
\begin{equation}
\xymatrix{
&&&\vt{1_1}\ar[drr]^-{a_{12}}\ar[d]_-{a_{11}}\ar[dll]_-{a_{1j}}&&\\
&\vt{j}\ar@/^/@(r,d)[]^{a_{jj}}\ar@<1mm>@/^2.8cm/[rrrr]^{a_{j2}}\ar@<1mm>[rr]^{a_{j1}}&&\vt{1_2}\ar@/^/@(r,d)[]^{a_{11}}\ar@<1mm>[ll]^{a_{1j}}\ar@<-1mm>[rr]_{a_{12}}&&\vt{2}\ar@/^/@(r,d)[]^{a_{22}}\ar@<-1mm>[ll]_{a_{21}}\ar@<1mm>@/_2.8cm/[llll]^{a_{2j}}\\
\bullet\ar[ur]_{c_j-a_{1j}}&&\bullet\ar[ur]_{c_1-a_{11}}&&\bullet\ar[ur]_{c_2-a_{12}}
}\label{prepared-reuse}
\end{equation}
we see that \vt{1_1} and \vt{1_2} emit identically and hence we can use move \IIIp (see comment just after Definition \ref{IIIp}). We move one edge that used to go from \vt{2} to \vt{1_2} so that it now goes to \vt{1_1}, and obtain
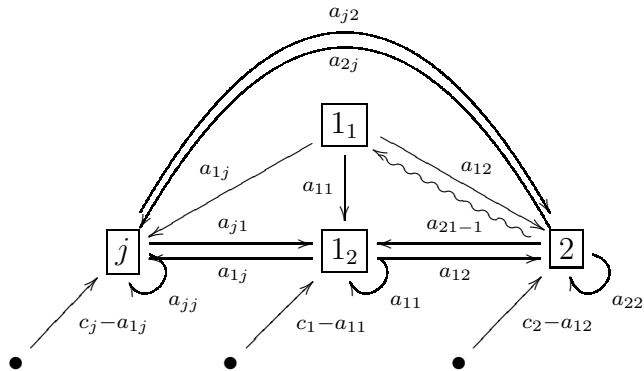
\[
\xymatrix{
&&&\vt{1_1}\ar@<1mm>[drr]^-{a_{12}}\ar[d]_-{a_{11}}\ar[dll]_-{a_{1j}}&&\\
&\vt{j}\ar@/^/@(r,d)[]^{a_{jj}}\ar@<1mm>@/^2.8cm/[rrrr]^{a_{j2}}\ar@<1mm>[rr]^{a_{j1}}&&\vt{1_2}\ar@/^/@(r,d)[]^{a_{11}}\ar@<1mm>[ll]^{a_{1j}}\ar@<-1mm>[rr]_{a_{12}}&&\vt{2}\ar@<1mm>@{~>}[ull]\ar@/^/@(r,d)[]^{a_{22}}\ar@<-1mm>[ll]_{a_{21-1}}\ar@<1mm>@/_2.8cm/[llll]^{a_{2j}}\\
\bullet\ar[ur]_{c_j-a_{1j}}&&\bullet\ar[ur]_{c_1-a_{11}}&&\bullet\ar[ur]_{c_2-a_{12}}
}
\]
There is no loop on \vt{1_1}, and it is regular, so we can use \RRRp\ and get
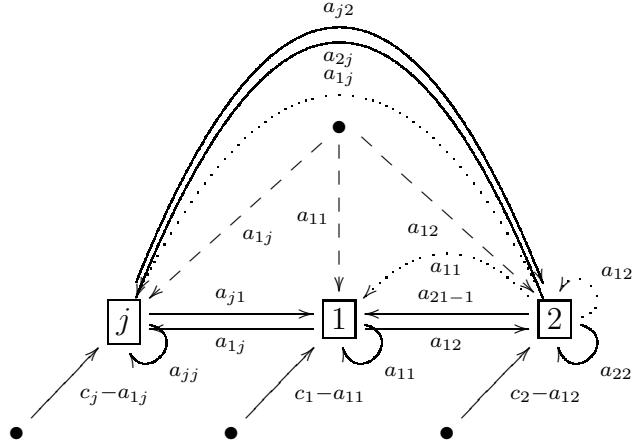
\[
\xymatrix{
&&&\bullet\ar@<1mm>@{-->}[ddrr]_-{a_{12}}\ar@{-->}[dd]_-{a_{11}}\ar@{-->}[ddll]^-{a_{1j}}&&\\&&&\\\
&\vt{j}\ar@/^/@(r,d)[]^{a_{jj}}\ar@<1mm>@/^3.8cm/[rrrr]^{a_{j2}}\ar@<1mm>[rr]^{a_{j1}}&&\vt{1}\ar@/^/@(r,d)[]^{a_{11}}\ar@<1mm>[ll]^-{a_{1j}}\ar@<-1mm>[rr]_{a_{12}}&&\vt{2}\ar@{..>}@<1mm>@/_10mm/[ll]^-{a_{11}}\ar@{..>}@/_30mm/[llll]_-{a_{1j}}\ar@/^/@(r,d)[]^{a_{22}}\ar@<-1mm>[ll]_{a_{21-1}}\ar@<1mm>@/_3.8cm/[llll]^{a_{2j}}\ar@(r,u)@{..>}[]_-{a_{12}}\\
\bullet\ar[ur]_{c_j-a_{1j}}&&\bullet\ar[ur]_{c_1-a_{11}}&&\bullet\ar[ur]_{c_2-a_{12}}
}
\]
which can be redrawn as
\[
\xymatrix{
&&&&&\\
&\vt{j}\ar@/^/@(r,d)[]^{a_{jj}}\ar@<1mm>@/^1.8cm/[rrrr]^{a_{j2}}\ar@<1mm>[rr]^{a_{j1}}&&\vt{1}\ar@/^/@(r,d)[]^{a_{11}}\ar@<1mm>[ll]^{a_{1j}}\ar@<-1mm>[rr]_{a_{12}}&&\vt{2}\ar@/^/@(r,d)[]^{a_{22}+a_{12}}\ar@<-1mm>[ll]_{a_{21}-1+a_{11}}\ar@<1mm>@/_1.8cm/[llll]^{a_{2j}+a_{1j}}\\
\bullet\ar[ur]_{c_j}&&\bullet\ar[ur]_{c_1}&&\bullet\ar[ur]_{c_2}
}
\]
after appropriate moves of type \OOO. By \eqref{ab} and \eqref{cd} we get:

\begin{propo}\label{col}
Given a pair  $\DBpair$ describing the graph $E$. When  $b_{12}>0$, $d_j\geq b_{j1}+1$ for $j>1$, and $d_1\geq b_{11}+2$, we can go from $E$ to  the graph described by the pair  $\DBpairprime$ given by 
\[
\DD'=\DD=\begin{pmatrix}d_1\\d_2\\d_3\\\vdots\\d_n\end{pmatrix}\quad \BB'=\begin{pmatrix}b_{11}&b_{12}+b_{11}&b_{13}&\dots&b_{1\n}\\
b_{21}&b_{22}+b_{21}&b_{23}&\dots&b_{2\n}\\
b_{31}&b_{32}+b_{31}&b_{33}&\dots&b_{3\n}\\\vdots&\vdots&\vdots&&\vdots\\b_{n1}&b_{n2}+b_{n1}&b_{n3}&\dots&b_{n\n}
\end{pmatrix}
\] 
by moves of type  \OOO, \IIIp, and \RRRp.
\end{propo}



%

\subsection{Insplitting gives column addition to antennae}
In this section we show how to increase the size of the $\DD$ vector in a pair $\DBpair$ without changing $\BB$.

\begin{propo}\label{antcol}
Given a pair $\DBpair$ describing the graph $E$.  When $b_{12}>0$, $d_j\geq b_{j1}+1$ for $j>1$, and $d_1\geq b_{11}+3$, we can go from $E$ to the  graph described by $\DBpairprime$ given by
\[
\DD'=\begin{pmatrix}d_1+b_{11}\\d_2+b_{21}\\d_3+b_{31}\\\vdots\\d_n+b_{n1}\end{pmatrix}\quad \BB'=\BB=\begin{pmatrix}b_{11}&b_{12}&b_{13}&\dots&b_{1\n}\\
b_{21}&b_{22}&b_{23}&\dots&b_{2\n}\\
b_{31}&b_{32}&b_{33}&\dots&b_{3\n}\\\vdots&\vdots&\vdots&&\vdots\\b_{n1}&b_{n2}&b_{n3}&\dots&b_{n\n}
\end{pmatrix}
\] 
by moves of type \OOO, \IIIp, and \RRRp.
\end{propo}
\begin{proof}
We proceed as in the previous section, but require further that $c_1>a_{11}$. Then when we get to the stage \eqref{prepared-reuse}
the middle source in the bottom supports at least one edge and hence has not been deleted. We use \IIIp\ as before, but as we may, we also move one edge from the shadow source of \vt 1 over to \vt{1_1}, so we get
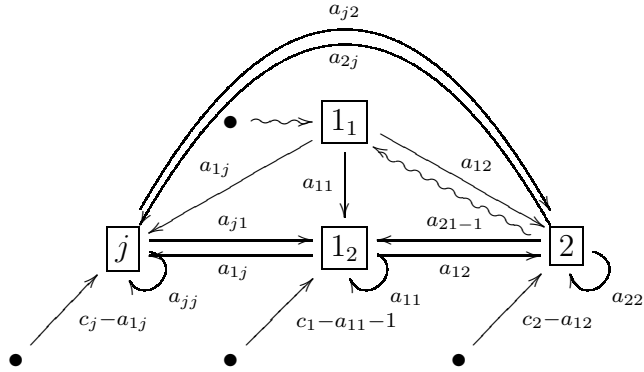
\[
\xymatrix{
&&\bullet\ar@{~>}[r]&\vt{1_1}\ar@<1mm>[drr]^-{a_{12}}\ar[d]_-{a_{11}}\ar[dll]_-{a_{1j}}&&\\
&\vt{j}\ar@/^/@(r,d)[]^{a_{jj}}\ar@<1mm>@/^2.8cm/[rrrr]^{a_{j2}}\ar@<1mm>[rr]^{a_{j1}}&&\vt{1_2}\ar@/^/@(r,d)[]^{a_{11}}\ar@<1mm>[ll]^{a_{1j}}\ar@<-1mm>[rr]_{a_{12}}&&\vt{2}\ar@<1mm>@{~>}[ull]\ar@/^/@(r,d)[]^{a_{22}}\ar@<-1mm>[ll]_{a_{21-1}}\ar@<1mm>@/_2.8cm/[llll]^{a_{2j}}\\
\bullet\ar[ur]_{c_j-a_{1j}}&&\bullet\ar[ur]_{c_1-a_{11}-1}&&\bullet\ar[ur]_{c_2-a_{12}}
}
\]
Now as we perform \RRRp\ there will have been indirect paths from the new source to everything else, the net effect being that the number of antennae arising from the \RRRp\ move is doubled, resulting in
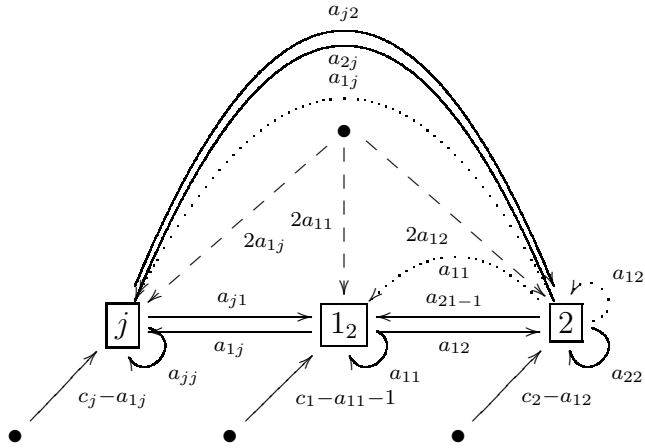
\[
\xymatrix{
&&&\bullet\ar@<1mm>@{-->}[ddrr]_-{2a_{12}}\ar@{-->}[dd]_-{2a_{11}}\ar@{-->}[ddll]^-{2a_{1j}}&&\\&&&\\\
&\vt{j}\ar@/^/@(r,d)[]^{a_{jj}}\ar@<1mm>@/^3.8cm/[rrrr]^{a_{j2}}\ar@<1mm>[rr]^{a_{j1}}&&\vt{1_2}\ar@/^/@(r,d)[]^{a_{11}}\ar@<1mm>[ll]^-{a_{1j}}\ar@<-1mm>[rr]_{a_{12}}&&\vt{2}\ar@{..>}@<1mm>@/_10mm/[ll]^-{a_{11}}\ar@{..>}@/_30mm/[llll]_-{a_{1j}}\ar@/^/@(r,d)[]^{a_{22}}\ar@<-1mm>[ll]_{a_{21-1}}\ar@<1mm>@/_3.8cm/[llll]^{a_{2j}}\ar@(r,u)@{..>}[]_-{a_{12}}\\
\bullet\ar[ur]_{c_j-a_{1j}}&&\bullet\ar[ur]_{c_1-a_{11}-1}&&\bullet\ar[ur]_{c_2-a_{12}}
}
\]
in which sources can be collected to form
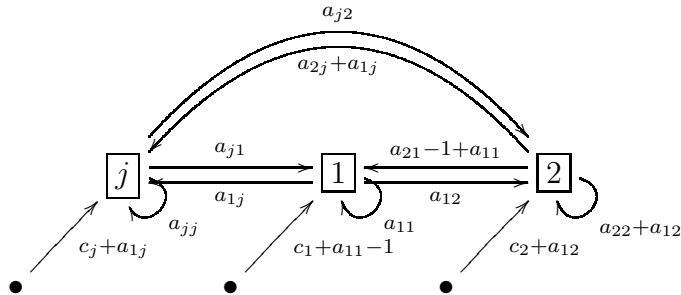
\[
\xymatrix{
&&&&&\\
&\vt{j}\ar@/^/@(r,d)[]^{a_{jj}}\ar@<1mm>@/^1.8cm/[rrrr]^{a_{j2}}\ar@<1mm>[rr]^{a_{j1}}&&\vt{1}\ar@/^/@(r,d)[]^{a_{11}}\ar@<1mm>[ll]^{a_{1j}}\ar@<-1mm>[rr]_{a_{12}}&&\vt{2}\ar@/^/@(r,d)[]^{a_{22}+a_{12}}\ar@<-1mm>[ll]_{a_{21}-1+a_{11}}\ar@<1mm>@/_1.8cm/[llll]^{a_{2j}+a_{1j}}\\
\bullet\ar[ur]_{c_j+a_{1j}}&&\bullet\ar[ur]_{c_1+a_{11}-1}&&\bullet\ar[ur]_{c_2+a_{12}}
}
\]
This graph is represented by $\DBpairprime$ given by
\[
\DD'=\begin{pmatrix}d_1+b_{11}\\d_2+b_{21}\\d_3+b_{31}\\\vdots\\d_n+b_{n1}\end{pmatrix}\quad \BB'=\begin{pmatrix}b_{11}&b_{12}+b_{11}&b_{13}&\dots&b_{1\n}\\
b_{21}&b_{22}+b_{21}&b_{23}&\dots&b_{2\n}\\
b_{31}&b_{32}+b_{31}&b_{33}&\dots&b_{3\n}\\\vdots&\vdots&\vdots&&\vdots\\b_{n1}&b_{n2}+b_{n1}&b_{n3}&\dots&b_{n\n}
\end{pmatrix}
\] 
which is exactly the result of applying  moves of type \OOO, \IIIp, and \RRRp\ as in Proposition \ref{col} to the graph described by $(\DD',\BB)$, noting that the requirements are met because of our assumptions. This proves the claim.
\end{proof}


\section{General matrix operations}

In this section we generalize the elementary matrix operations of the previous section to much more general settings under rather modest conditions on the graphs studied. We also discuss in this context the possibility of performing the operations in reverse, as row or column subtractions rather than additions.

We generalize row addition before specifying conditions, as it will be convenient to prove that the conditions are always obtainable after moves of type \OOO, \IIIp\ or \RRRp.

\subsection{Improved row addition}

\begin{propo}\label{rowplus}
Given a pair $\DBpair$ describing the graph $E$. When two different vertices $\vt{i}$ and $\vt{j}$ are given so that $\vt{i}$ supports a loop, and there is a path from $\vt{i}$ to $\vt{j}$,  then we can go from $E$ to the graph described by the pair $\DBpairprime$ given
\[    
\DD'=\begin{pmatrix}\vdots\\d_{j-1}\\d_j+d_i\\d_{j+1}\\\vdots\end{pmatrix}\quad \BB'=\begin{pmatrix}
\vdots&\vdots&\vdots&&\vdots\\
b_{j-1,1}&b_{j-1,2}&b_{j-1,3}&\dots&b_{j-1,\n}\\
b_{j1}+b_{i1}&b_{j2}+b_{i2}&b_{j3}+b_{i3}&\dots&b_{j\n}+b_{i\n}\\
b_{j+1,1}&b_{j+1,2}&b_{j+1,3}&\dots&b_{j+1,\n}\\
\vdots&\vdots&\vdots&&\vdots
\end{pmatrix}
\] 
by  moves of type \OOO\ and \RRRp.
\end{propo}
\begin{proof}
We may assume that $i=1$ and that there is a minimal path from $\vt{1}$ to $\vt{j}$ passing through $\vt{2},\dots,\vt{j-1}$ in order. Hence $a_{\ell,\ell+1}=b_{\ell+1,\ell}>0$.
Our argument depends on whether  the intermediate $1< \ell<j$ fall in the set
\[
S=\left\{\ell\in\{1,\dots, j-1\}\left|\sum_{i=1}^n b_{i\ell}=0\right\}\right.
\]
or not. Importantly, $1\not \in S$ by our assumption that $b_{11}\geq 0$ and $b_{21}>0$.

We first consider $j-1$, noting that $b_{j,j-1}>0$. If  $j-1\in S$ we know that $\vt{j-1}$ emits exactly one edge, namely to \vt{j}, and hence the setup is
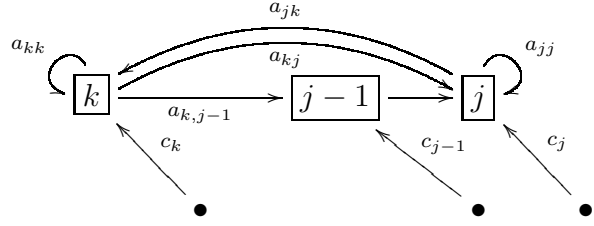
\[
\xymatrix{\vt{k}\ar@(u,l)[]_{a_{kk}}\ar@<-1mm>@/^2em/[rrr]_-{a_{kj}}\ar[rr]_-{a_{k,j-1}}&&\vt{j-1}\ar[r]&\vt{j}\ar@(u,r)[]^-{a_{jj}}\ar@<-1mm>@/_2em/[lll]_-{a_{jk}}&\\
&\bullet\ar[ul]_-{c_k}&&\bullet\ar[ul]_-{c_{j-1}}&\bullet\ar[ul]_-{c_j}}
\]
(generic $k\in\{1,\dots,n\}\backslash\{j-1,j\}$) which becomes 
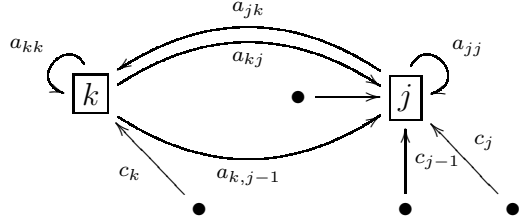
\[
\xymatrix{\vt{k}\ar@(u,l)[]_{a_{kk}}\ar@<-1mm>@/^2em/[rrr]_-{a_{kj}}\ar@/_2em/[rrr]_-{a_{k,j-1}}&&\bullet\ar[r]&\vt{j}\ar@(u,r)[]^-{a_{jj}}\ar@<-1mm>@/_2em/[lll]_-{a_{jk}}&\\
&\bullet\ar[ul]^-{c_k}&&\bullet\ar[u]_-{c_{j-1}}&\bullet\ar[ul]_-{c_j}}
\]
after an \RRRp\ move. This is represented by the pair
\[    
\left(\begin{pmatrix}d_1\\\vdots\\d_{j-2}\\d_{j-1}+d_j\\d_{j+1}\\\vdots\end{pmatrix}, \begin{pmatrix}
b_{11}&\cdots&b_{1,j-2}&b_{1j}&\cdots\\
\vdots&&\vdots&\vdots&\\
b_{j-2,1}&\cdots&b_{j-2,j-2}&b_{j-2,j}&\cdots\\
b_{j1}+b_{j-1,1}&\cdots&\underline{b_{j,j-2}+b_{j-1,j-2}}&b_{jj}+b_{j-1,j}&\dots\\
b_{j+1,1}&\cdots&b_{j+1,j-2}&b_{j+1,j}&\cdots\\
\vdots&&\vdots&\vdots
\end{pmatrix}\right)
\] 
where the $(j-1)$st column has also been deleted. When $j-1\not\in S$, we note that Proposition \ref{row} applies, 
and arrive at
\[    
\scalebox{.9}{$
\left(\begin{pmatrix}d_1\\\vdots\\d_{j-2}\\d_{j-1}\\d_{j-1}+d_j\\\vdots\end{pmatrix}, \begin{pmatrix}b_{11}&\cdots&b_{1,{j-2}}&b_{1,j-1}&b_{1j}&\dots\\
\vdots&&\vdots&\vdots&\vdots&\\
b_{j-2,1}&\cdots&b_{j-2,j-2}&b_{j-2,j-1}&b_{j-2,j}&\cdots\\
b_{j-1,1}&\cdots&b_{j-1,j-2}&b_{j-1,j-1}&b_{j-1,j}&\cdots\\
b_{j1}+b_{{j-1},1}&\cdots&\underline{b_{j,j-2}+b_{j-1,j-2}}&b_{j,j-1}+b_{j-1,j-1}&b_{jj}+b_{j-1,j}&\cdots\\
\vdots&&\vdots&\vdots&\vdots
\end{pmatrix}\right)$}
\] 
In either case, we may use our knowledge that $b_{j-1,j-2}>0$ to see that we may now use one of these operations to add row $j-2$ to row $j$ (because of the nonzero entry at the underlined entries) and we can continue this way until we reach the pair $\DBpairprimeprime$ obtained by replacing the $j$th row in $\DBpair$ by
\[
(\begin{pmatrix}\sum_{\ell=1}^j d_\ell\end{pmatrix}, \begin{pmatrix}\sum_{\ell=1}^j b_{\ell 1}&\cdots&\sum_{\ell=1}^j b_{\ell,j-2}&\sum_{\ell=1}^j b_{\ell,j-1}&\sum_{\ell=1}^j {b_{\ell,j}}&\cdots\end{pmatrix})
\]
and then deleting all rows and columns corresponding to entries in $S$. We note that by performing only the $j-2$ first such steps, but starting from $\DBpairprime$, we also get to $\DBpairprimeprime$, proving the claim.
\end{proof}
\subsection{Augmented canonical form}

In this section we describe the notion of \emph{augmented canonical form} and explain how it may be algorithmically arranged. The notion is the direct extension of \emph{canonical form} from \cite{segrerapws:gcgcfg} and \cite{segrerapws:ccuggs} to the antenna calculus setting, but in keeping with the more geometrical nature of this work, we will focus on graphs rather than matrices when specifying it. This may even be a useful perspective for a reader of the aforementioned papers, and we indicate in Remark \ref{sameaserrs} how the notions coincide.

As in \cite{segrerapws:gcgcfg}, we denote by $\gamma(i)$ the \emph{component} associated to a vertex \vt{i} as the largest set of vertices in $\{1,\dots,n\}$ so that $i\in\gamma(i)$ and so that whenever $j,k\in\gamma(i)$ are different, there is a path from \vt{j} to \vt{k} (and back by symmetry). We denote by $|\gamma(i)|$ the number of elements, divided
\[
|\gamma(i)|=|\gamma(i)|^\bullet+|\gamma(i)|^\circ
\]
into regular and singular vertices if necessary. Note that when $|\gamma(i)|>1$, there is always a path from \vt i back to itself, but that this is not always the case when $|\gamma(i)|=1$. We denote the set of components by $\Gamma_E$ and note that it is pre-ordered because we may say that $\gamma(i)\leq  \gamma(j)$ when there is a path from \vt j to \vt i (or when $\gamma(i)=\gamma(j)$).

Using a topological ordering of $\Gamma_E$ we may and shall think of $\AA$ and $\BB$ as block triangular matrices with the vertices ordered so that each component corresponds to a segment $i,\dots,i+k$. We work exclusively with the $\BB$, and denote the block corresponding to rows from $\gamma(i)$ and columns from $\gamma(j)$  by $\BB_{\gamma(i),\gamma(j)}$. We write $\BB_{\gamma(i)}$ for the diagonal blocks $\BB_{\gamma(i),\gamma(i)}$.
When we write $\XX\in \MG[E](\nn)$ when $\XX_{\gamma(i),\gamma(j)}=0$ unless $\gamma(i)\leq \gamma(j)$, and where $\nn$ is a vector indicating the size of each block, it is clear that $\BB\in \MG[E](\nn)$ for $\nn=(|\gamma(i)|)_{\gamma(i)\in \Gamma_E}$.
\newcommand{\mc}{\operatorname{mr}}

We associate the number $\mc(\gamma(i))$ to any component $\gamma(i)$, defined as $k+\ell$, where 
\[
\cok \BB^\bullet_{\gamma(i)}\simeq \ZZ/d_1\oplus\cdots \oplus  \ZZ/d_k\oplus \ZZ^\ell
\]
and $d_j|d_{j+1}$ for all $i$. 
It follows from $K$-theory that this number is invariant under all our moves, and in a certain sense expresses the smallest number of vertices that represents this component. We will not go into this since the condition only enters indirectly via  \cite{segrerapws:ccuggs}.


\begin{defin}\label{def:CanonicalForm}
Let $E$ be a given graph.
We say that $E$ is in \emph{augmented canonical form} if 
\begin{enumerate}[\rm (I)]
\item \label{def:canonical-item-loops} 
every regular vertex of $E$ which is not a source supports a loop;
\item \label{def:canonical-item-paths} whenever there is a path from \vt i to \vt j, there is an edge from \vt i to \vt j;
\item \label{def:canonical-item-infemit} whenever there is a path from \vt i to \vt j, and \vt i is an infinite emitter, there are infinitely many edges from \vt i to \vt j;
\item \label{def:canonical-item-singularfirst}
If there are two different paths from \vt i back to itself (neither visiting \vt i along the way), then \vt i supports  two loops, and 
 $|\gamma(i)|^\bullet\geq \max\{3,\mc(\gamma(i))+2\}$.
\end{enumerate}
If $E$ has no regular sources, we just say that $E$ is in canonical form.
\end{defin}

Note from the outset that there is a trichotomy among components in a graph that is in augmented canonical form. If one (hence all) of the vertices in the component has more than one path back to itself, then there are more than a prescribed number of regular vertices in the component by (IV), all vertices are directly connected by (II), and every vertex supports two loops by (IV) again. If this is not the case, there is only one vertex in the component because of (I). If this vertex is regular, it supports exactly one loop, and if it is singular, it has no path back to itself and hence there are no edges in the component. We also get from (III) that any infinite emitter emits with infinite multiplicity to any vertex it reaches by a path.

It is worth translating these conditions to the pair $\DBpair$, and we see first that they only involve $\BB$.
The  diagonal blocks correspond to the components themselves, and we see that they come in three flavors: One type is ``large'' and contains only positive entries by (II) and (IV) (they are even $\infty$ in all columns corresponding to singular vertices by (II)), and the other two types are ``small'' and must be one of
\[
\begin{pmatrix}0\end{pmatrix}\qquad \begin{pmatrix}-1\end{pmatrix}.
\] 
All off-diagonal blocks that are allowed to take nonzero entries by the condition defining $\MG[E](\nn)$ are in fact positive everywhere because of (II), even $\infty$ on all singular columns by (III).

\begin{propo}\label{getcanon}
Any graph $E$ may be transformed algorithmically into a graph $E'$ in augmented canonical form by moves of type \OOO, \IIIp, and \RRRp\ as follows:
\begin{enumerate}[\textsc{Step} 1]
\item Use \OOO\ moves to ensure that if \vt{i} is an infinite emitter, then it emits either infinitely many or no edges to any vertex. \item Use \RRRp\ moves to ensure that any regular vertex not supporting a loop is a source.
\item Use \OOO\ moves to ensure that  no component has only one vertex but two or more edges, so that the properties of \textsc{Step} 1 and 2 are preserved. 
\item Use improved row addition (Proposition \ref{rowplus}) to ensure that all entries in the $\BB_{\gamma(i)}$-block for any component $\gamma(i)$ with  $|\gamma(i)|>1$ are positive.
\item Use an \OOO\ move to increase the size of any component not satisfying (IV) by one, so that the properties of \textsc{Step} 1 and 2 are preserved. Go back to \textsc{Step} 4 if the  $\BB_{\gamma(i)}$-block of the outsplit graph is not positive for any $\gamma(i)$ with $|\gamma(i)|<1$.
\item Use improved row addition (Proposition \ref{rowplus}) to ensure that $b_{ij}>0$ whenever $\gamma(i)\not=\gamma(j)$, $\gamma(i)\leq\gamma(j)$  and $|\gamma(j)|>1$.
\item Use improved row addition (Proposition \ref{rowplus}) to ensure that $b_{ij}>0$ whenever $\gamma(i)\not=\gamma(j)$, $\gamma(i)\leq\gamma(j)$ 
$\BB_{\gamma(i),\gamma(j)}\not=\OO$.
\item If three vertices \vt i, \vt j and \vt k are given with an edge from \vt i to \vt j and an edge from \vt j to \vt k, but no edge from \vt i to \vt k, then use basic row addition (Proposition \ref{row}) to add row $j$ to row $k$. Then return to \textsc{Step} 7.
\end{enumerate}
 \end{propo}
\begin{proof}
\textsc{Step} 1 is obtained by placing all edges parallel to infinitely many edges in one set of the partition, and the rest in the other.
 We also note that the properties arranged in \textsc{Step} 1 and 2 will not be affected in later steps, since care is taken when applying subsequent \OOO\ moves and since row additions will never change this property.
 
 To see that  \textsc{Step} 3 is possible, note that we can replace $\begin{bmatrix}n\end{bmatrix}$ by $\left[\begin{smallmatrix}1&1\\n-1&n-1\end{smallmatrix}\right]$ and  $\begin{bmatrix}\infty\end{bmatrix}$ by $\left[\begin{smallmatrix}1&1\\\infty&\infty\end{smallmatrix}\right]$. In the latter case, we must assign all edges leaving the component to the second vertex to preserve \textsc{Step} 1. Similar arguments apply in \textsc{Step} 5.

In \textsc{Step} 4, we sum all rows into the last row and see that produces exclusively positive entries. This may then be added to all other rows. Proposition \ref{rowplus} applies because of  \textsc{Step} 2 and \textsc{Step} 3. In \textsc{Step} 6, the proposition applies because any vertex in the emitting component supports a loop, and gives positive entries in the off-diagonal component since the diagonal entry in the added row is postive (because it in fact supports two loops). In \textsc{Step} 7, we may assume by \textsc{Step} 6 that there is only one column in the block, so the given nonzero entry can be used to make all entries positive by row operations among the vertices in $\gamma(i)$. Proposition \ref{rowplus} applies when $|\gamma(i)|>1$, and when $|\gamma(i)|=1$ there is nothing to do.

In  \textsc{Step} 8,  we see by the previous steps that \vt i, \vt j and \vt k lie in three different components, and that $|\gamma(i)|=1$ so that the block $\BB_{\gamma(i),\gamma(k)}$ has only one column. We may apply Proposition \ref{row} because there is an edge from \vt j to \vt k, and because \vt j emits at least one other edge. Indeed, it will either support two loops, one loop, or be an infinite emitter depending on which element in the trichotomy it belongs to.

The algorithm clearly terminates, and it follows that (I)--(IV) are satisfied in the resulting graph.
\end{proof}

\begin{remar}\label{subtract}
Although we do not require graphs to be on augmented canonical form before performing row additions, it is relevant for performing row \emph{subtractions}. Indeed, it is obvious that when we can go from $\DBpair$ to $\DBpairprime$ by a row addition implemented by moves of type \OOO, \IIIp, and \RRRp, we can go from $\DBpairprime$ to $\DBpair$ by such moves as well. Starting from $\DBpairprime$, we just need to ensure that such an operation does not introduce entries inconsistent with the way we represent graphs: $d_i$ must be at least one, $b_{ij}$ must be nonnegative for $i\not=j$, and $b_{ii}\geq -1$.

Computing $\DBpairprime$ from $\DBpair$ is problematic in columns with infinite emitters, since $\infty-\infty$ is undefined, but it follows from (II) that any row addition in the presence of augmented canonical form does not alter such a column. Hence it makes sense to use the convention $\infty-\infty=\infty$ in this case. 
\end{remar}

\begin{remar}\label{sameaserrs}
The condition  (I) studied here implies the conditions defining $\MPZcc$ and $\MPZccc$ as well as the first half of (1) in the definition of canonical form in  \cite{segrerapws:gcgcfg,segrerapws:ccuggs}. Our (II) similarly gives $\MPZc$ and the second half of (1), and our  (III) gives (4) of canonical form in \cite{segrerapws:gcgcfg,segrerapws:ccuggs}. (IV) gives the remaining conditions (2), (3), and (5).
\end{remar}

\subsection{Increasing antenna counts}

The following result is of key technical importance for us. We will use it to increase the number of antennae to suit our needs, in particular when generalizing the column operation from Proposition \ref{col} to a much more general version. Employing an assumption of augmented standard form allows us to show that we may increase antenna counts like this in any such setting. 

\begin{theor}\label{addcolalways}
Let $E$ be a graph  in augmented canonical form represented by $\DBpair$. For any $j$ with \vt{j} regular, we can go to the graph described by the pair $\DBpairprime$ given by
\[
\DD'=\begin{pmatrix}d_1+b_{1j}\\d_2+b_{2j}\\d_3+b_{3j}\\\vdots\\d_n+b_{nj}\end{pmatrix}\quad \BB'=\BB=\begin{pmatrix}b_{11}&b_{12}&b_{13}&\dots&b_{1\n}\\
b_{21}&b_{22}&b_{23}&\dots&b_{2\n}\\
b_{31}&b_{32}&b_{33}&\dots&b_{3\n}\\\vdots&\vdots&\vdots&&\vdots\\b_{n1}&b_{n2}&b_{n3}&\dots&b_{n\n}
\end{pmatrix}
\] 
by  moves of type \OOO,  \IIIp, and \RRRp.
\end{theor}
%

\begin{proof}
We assume without loss of generality that $j=1$ but note for later use that Proposition \ref{antcol} allows us to add any regular column, say $i$, to $\DD$ provided $d_j\geq b_{ji}+1$ for $j\not=i$, and $d_i\geq b_{ii}+3$, when there is some $j\not=i$ so that $b_{ij}>0$.


\thecase{1}  Suppose $b_{11} = b_{12} = \cdots = b_{1\n} = 0$.  Then \vt{1} supports a single loop and besides this loop receives only from a regular source. When also $ b_{21} = \cdots = b_{\n1} = 0$ there is nothing to prove, so we assume that \vt 1 emits to at least two vertices.
By repeated use of \RRRp\ in reverse we pass to the graph
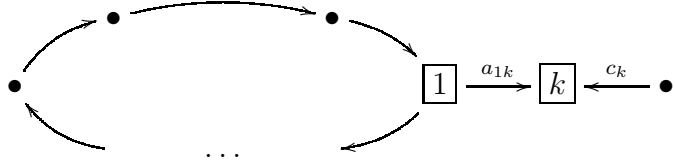
\[
\xymatrix@R=0.3cm{&\bullet\ar@/^/[rr]&&\bullet\ar@/^/[dr]&&\\
\bullet\ar@/^/[ur]&&&&\vt{1}\ar[r]^-{a_{1k}}\ar@/^/[dl]&\vt{k}&\bullet\ar[l]_-{c_k}\\
&\ar@/^/[ul]&\dots&&}
\]
with the loop of length $d_1=1+c_1$, having the important property that there is exactly one incoming edge to $\vt{1}$. We now use \OOO\ at \vt{1} (using here that it emits more than one edge) and get
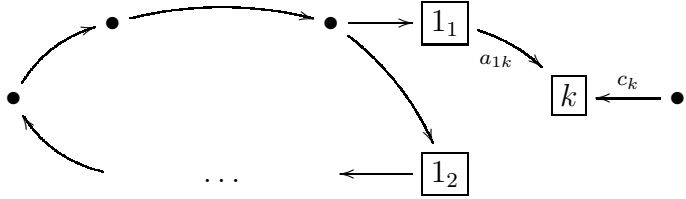
\[
\xymatrix@R=0.2cm{&\bullet\ar@/^/[rr]&&\bullet\ar[r]\ar@/^/[ddr]&\vt{1_1}\ar@/^/[dr]_-{a_{1k}}\\
\bullet\ar@/^/[ur]&&&&&\vt{k}&\bullet\ar[l]_-{c_k}&&\\
&\ar@/^/[ul]&\dots&&\vt{1_2}\ar[l]}
\]
and then with \RRRp\ at $\vt{1_1}$ we get
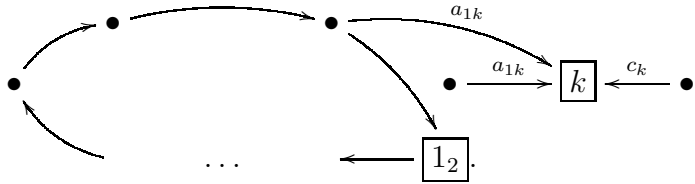
\[
\xymatrix@R=0.2cm{&\bullet\ar@/^/[rr]&&\bullet\ar@/^3.5mm/[drr]^{a_{1k}}\ar@/^/[ddr]&&\\
\bullet\ar@/^/[ur]&&&&\bullet\ar[r]^-{a_{1k}}&\vt{k}&\bullet\ar[l]_-{c_k}\\
&\ar@/^/[ul]&\dots&&\vt{1_2}.\ar[l]}
\]
(some $a_{1k}$ may be zero, but not all). Shortening the loop again with \RRRp\ moves, we arrive at the desired situation. This works also when $c_1=0$ but takes the form
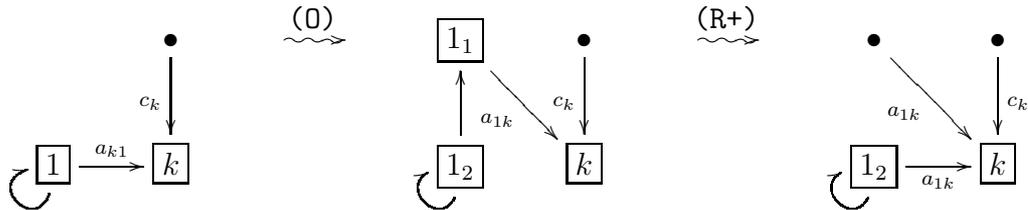
\[
\xymatrix{
&\bullet\ar[d]_{c_k}&\ar@{~>}[r]^{\OOO}&&\vt{1_1}\ar[dr]_-{a_{1k}}&\bullet\ar[d]_{c_k}&\ar@{~>}[r]^{\RRRp}&&\bullet\ar[dr]_-{a_{1k}}&\bullet\ar[d]^{c_k}&\\
\vt{1}\ar@(d,l)[]\ar[r]^{a_{k1}}&\vt{k}&&&\vt{1_2}\ar@(d,l)\ar[u]&\vt{k}&&&\vt{1_2}\ar@(d,l)\ar[r]_-{a_{1k}}&\vt{k}&
}
\]

\thecase{2} Suppose $b_{11}  = 0$ but that the first row does not vanish. We may assume that $b_{12}>0$. By (II) and (IV) of our assumption of augmented canonical form, \vt{1} is alone in its component, so we conclude that $b_{21} = 0$.  Since $b_{12} \geq 1$, we may apply Proposition \ref{row} and add row $2$ to row $1$ twice and get to 
\[
\mypair{\begin{pmatrix}d_1 + 2d_2 \\d_2 \\d_3\\\vdots\\d_n \end{pmatrix}}{ \begin{pmatrix}0 &b_{12} + 2b_{22} &b_{13} + 2b_{23} &\dots&b_{1k} + 2 b_{2\n}\\
b_{21}&b_{22}&b_{23}&\dots&b_{2\n}\\
b_{31}&b_{32}&b_{33}&\dots&b_{3\n}\\\vdots&\vdots&\vdots&&\vdots\\b_{n1}&b_{n2}&b_{n3}&\dots&b_{n\n}
\end{pmatrix}}
\] 
Since $d_1,d_2>0$ we can choose $M_j\geq0$ so that
\begin{equation}\label{prepareforcol}
d_ j + M_j( d_1 + 2d_2 )  \geq b_{j1}+1
\end{equation}
for all $j$ with $b_{j1} > 0$.   For $j$ with $b_{j1} = 0$, set $M_j = 0$. For all $j$, add row 1 to row $j$, $M_j$ times to get to
\[
\scalebox{.8}{
$\mypair{\begin{pmatrix}d_1 + 2d_2 \\d_2 + M_2( d_1 + 2d_2 )  \\d_3 + M_3( d_1 + 2d_2 )\\\vdots\\d_n + M_n( d_1 + 2d_2 ) \end{pmatrix}}{ \begin{pmatrix}0 &b_{12} + 2b_{22} &b_{13} + 2b_{23} &\dots&b_{1\n} + 2 b_{2\n}\\
b_{21}  &b_{22} + M_2( b_{12} + 2b_{22} ) &b_{23} + M_2( b_{13} + 2b_{23} )&\dots&b_{2\n} + M_2( b_{1\n} + 2b_{2\n} )\\
b_{31}&b_{32}+  M_3( b_{12} + 2b_{22} ) &b_{33} + M_3( b_{13} + 2b_{23} ) &\dots&b_{3k} + M_3( b_{1\n} + 2b_{2\n} ) \\
\vdots&\vdots&\vdots&&\vdots\\
b_{n1}&b_{n2} + M_n( b_{12} + 2b_{22} ) &b_{n3} + M_n( b_{13} + 2b_{23} ) &\dots&b_{nk} + M_n( b_{1\n} + 2b_{2\n} )
\end{pmatrix}}
$
}
\] 
and note that since $b_{12} > 0$ and \eqref{prepareforcol} hold, Proposition \ref{antcol} applies to take us to
\[
\scalebox{.8}{
$\mypair{\begin{pmatrix}d_1 + 2d_2 \\d_2 + M_2( d_1 + 2d_2) + b_{21} \\d_3 + M_3( d_1 + 2d_2 ) + b_{31} \\\vdots\\d_n + M_n( d_1 + 2d_t )  + b_{n1} \end{pmatrix}}{ \begin{pmatrix}0 &b_{12} + 2b_{22} &b_{13} + 2b_{23} &\dots&b_{1\n} + 2 b_{2\n}\\
b_{21}  &b_{22} + M_2( b_{12} + 2b_{22} ) &b_{23} + M_2( b_{13} + 2b_{23} )&\dots&b_{2k} + M_2( b_{1\n} + 2b_{2\n} )\\
b_{31}&b_{32}+  M_3( b_{12} + 2b_{22} ) &b_{33} + M_3( b_{13} + 2b_{23} ) &\dots&b_{3k} + M_3( b_{1\n} + 2b_{2\n} ) \\
\vdots&\vdots&\vdots&&\vdots\\
b_{n1}&b_{n2} + M_n( b_{12} + 2b_{22} ) &b_{n3} + M_n( b_{13} + 2b_{23} ) &\dots&b_{nk} + M_n( b_{1\n} + 2b_{2\n} )
\end{pmatrix}}
$
}
\] 
and then for all $j$, subtracting row 1 from row $j$ $M_j$ times, we get to 
\[
\mypair{\begin{pmatrix}d_1 + 2d_2  \\d_2 + b_{21} \\d_3 + b_{31} \\\vdots\\d_n  + b_{n1} \end{pmatrix}}{ \begin{pmatrix}0 & b_{12} + 2b_{22} &b_{13} + 2b_{23} &\dots&b_{1\n} + 2 b_{2\n}\\
b_{21}  &b_{22}  &b_{23} &\dots&b_{2\n} \\
b_{31}&b_{32} &b_{33} &\dots&b_{3\n}   \\
\vdots&\vdots&\vdots&&\vdots\\
b_{n1}&b_{n2}  &b_{n3}  &\dots&b_{n\n} 
\end{pmatrix}}
\] 
Finally, recall that $b_{21} = 0$, thus subtracting row $2$ from row 1 twice, we get to 
\[
\mypair{\begin{pmatrix}d_1 + 0  \\d_2 + b_{21} \\d_3 + b_{31} \\\vdots\\d_n  + b_{n1} \end{pmatrix}}{ \begin{pmatrix}0 & b_{12} &b_{13}  &\dots&b_{1\n} \\
b_{21}  &b_{22}  &b_{23} &\dots&b_{2\n} \\
b_{31}&b_{32} &b_{33} &\dots&b_{3\n}   \\
\vdots&\vdots&\vdots&&\vdots\\
b_{n1}&b_{n2}  &b_{n3}  &\dots&b_{n\n}
\end{pmatrix}}
\]
by  a succession of  moves of type \OOO, \IIIp, and \RRRp.

\thecase{3} Since $E$ is in augmented canonical form,  the remaining case has  $b_{11} > 0$, and  we may assume that $b_{12} > 0$ and $b_{21} > 0$ for some regular \vt 2.   Outsplitting $\vt 1$ using  a single loop on \vt{1} in one set of the partition, and the rest of the outgoing edges in the other, we get to
\[
\mypair{\begin{pmatrix}
d_1 \\ d_1 \\ d_2 \\d_3\\ \vdots \\ d_n 
\end{pmatrix}
}{
\begin{pmatrix}
0  &   b_{11}  & b_{12}  & b_{13} & \cdots & b_{1\n} \\
1 & b_{11}-1 &b_{12}&b_{13}&\dots&b_{1\n}\\
0 & b_{21}  &b_{22}&b_{23}&\dots&b_{2\n}\\
0 & b_{31} &b_{32}&b_{33}&\dots&b_{3\n}\\
\vdots & \vdots&\vdots&\vdots&&\vdots\\
0 & b_{n1}  &b_{n2}&b_{n3}&\dots&b_{n\n} 
\end{pmatrix}}
\]
by an \OOO\ move.

We claim that we can get to
\[
\mypair{\begin{pmatrix}
d_1 \\ d_1+ N  \\ d_2 \\ d_3\\\vdots \\ d_n 
\end{pmatrix}
}{
\begin{pmatrix}
0  &   b_{11}  & b_{12}  & b_{13} & \cdots & b_{1\n} \\
1 & b_{11}-1 &b_{12}&b_{13}&\dots&b_{1\n}\\
0 & b_{21}  &b_{22}&b_{23}&\dots&b_{2\n}\\
0 & b_{31} &b_{32}&b_{33}&\dots&b_{3\n}\\
\vdots & \vdots&\vdots&\vdots&&\vdots\\
0 & b_{n1}  &b_{n2}&b_{n3}&\dots&b_{n\n} 
\end{pmatrix}}
\]
by moves of type \IIIp, \OOO\ and \RRRp\  for all $N \geq 1$.

Since the $(2,1)$-entry of the above matrix is 1, we may add row $1$ to row $2$ twice by Proposition \ref{row} and get to
\[
\mypair{\begin{pmatrix}
d_1 \\ 3d_1 \\ 
\vdots 
\end{pmatrix}
}{
\begin{pmatrix}
0  &   b_{11}  & b_{12}  & b_{13} & \cdots & b_{1\n} \\
1 & 3b_{11}-1 &3b_{12}&3b_{13}&\dots&3b_{1\n}\\
\vdots & \vdots&\vdots&\vdots&&\vdots\\
\end{pmatrix}}
\]
where we skip all unaltered lines to conserve space.   Adding row 2 to row 1 (applying Proposition \ref{row} since the $(1,2)$-entry of the above matrix is $b_{11} > 0$), we get to  
\[
\mypair{\begin{pmatrix}
4d_1 \\ 3d_1 \\
 \vdots \\ 
\end{pmatrix}
}{
\begin{pmatrix}
1  &   4b_{11}-1  & 4b_{12}  & 4b_{13} & \cdots & 4b_{1\n} \\
1 & 3b_{11}-1 &3b_{12}&3b_{13}&\dots&3b_{1\n}\\
\vdots & \vdots&\vdots&\vdots&&\vdots\\
\end{pmatrix}}
\]
 By Proposition \ref{antcol} which applies because the two first entries in the vector dominate appropriately, and because the $(1,2)$-entry in the matrix is not zero, we get to 
\[
\mypair{\begin{pmatrix}
4d_1 + N \\ 3d_1 + N \\ 
\vdots 
\end{pmatrix}
}{
\begin{pmatrix}
1  &   4b_{11}-1  & 4b_{12}  & 4b_{13} & \cdots & 4b_{1\n} \\
1 & 3b_{11}-1 &3b_{12}&3b_{13}&\dots&3b_{1\n}\\
\vdots & \vdots&\vdots&\vdots&&\vdots\\
\end{pmatrix}}
\]
and subtracting row 2 from row 1, we get to 
\[
\mypair{\begin{pmatrix}
d_1 \\ 3d_1  + N \\
\vdots 
\end{pmatrix}
}{
\begin{pmatrix}
0  &   b_{11}  & b_{12}  & b_{13} & \cdots & b_{1\n} \\
1 & 3b_{11}-1 &3b_{12}&3b_{13}&\dots&3b_{1\n}\\
\vdots & \vdots&\vdots&\vdots&&\vdots\\
\end{pmatrix}}.
\]
 Subtracting row 1 from row 2 twice, we get to 
\begin{equation}\label{putN}
\mypair{\begin{pmatrix}
d_1 \\ d_1+N\\ 
 \vdots 
\end{pmatrix}
}{
\begin{pmatrix}
0  &   b_{11}  & b_{12}  & b_{13} & \cdots & b_{1\n} \\
1 & b_{11}-1 &b_{12}&b_{13}&\dots&b_{1\n}\\
\vdots & \vdots&\vdots&\vdots&&\vdots\\
\end{pmatrix}}
\end{equation}
by moves of type \OOO, \IIIp, and \RRRp, as claimed.

Choose $N \in \NN$ such that 
\begin{align}\label{dominate}
d_j + d_1 + N + 1 &\geq b_{j1 } + b_{11}-1
\end{align}
for all $j\geq 1$ with $b_{j1} > 0$ (recall that $b_{j1}<\infty$ since \vt{1} is regular).  We now set $\Delta_j=1$ when $b_{j1}>0$ and  $\Delta_j=0$ otherwise.
%

Adding row $2$ to row $1$ to \eqref{putN}, as well as adding row 2 to row $j$ for all $j$ with $b_{j1} > 0$, we get to 
\[
\scalebox{.85}{$\mypair{\begin{pmatrix}
2d_1 + N \\ d_1+N \\ d_2 +\Delta _2(d_1 + N) \\ \vdots \\ d_n + \Delta _n(d_1 + N)
\end{pmatrix}
}{
\begin{pmatrix}
1  &   2b_{11} - 1& 2b_{12}  & 2b_{13} & \cdots & 2b_{1\n} \\
1 & b_{11}-1 &b_{12}&b_{13}&\dots&b_{1\n}\\
\Delta_2& b_{21}+\Delta_2(b_{11}-1)  &b_{22}+\Delta_2b_{12} &b_{23}+\Delta_2b_{13}&\dots&b_{2\n}+\Delta_2b_{1\n}\\
\vdots & \vdots&\vdots&\vdots&&\vdots\\
\Delta_n& b_{n1}+\Delta_n(b_{11}-1)  &b_{n2}+\Delta_nb_{12} &b_{n3}+\Delta_nb_{13}&\dots&b_{n\n}+\Delta_nb_{1\n}\\
\end{pmatrix}}.$}
\]
Applying Proposition \ref{antcol}, which applies because of \eqref{dominate}, we now get to
\[
\scalebox{.75}{$\mypair{\begin{pmatrix}
2d_1 + N+ 2b_{11}-1  \\ d_1+N  + b_{11}-1\\ d_2 +b_{21}+\Delta _2(d_1 + N+b_{11}-1) \\ \vdots \\ d_n + b_{31}+\Delta _n(d_1 + N+b_{11}-1)
\end{pmatrix}
}{
\begin{pmatrix}
1  &   2b_{11} - 1& 2b_{12}  & 2b_{13} & \cdots & 2b_{1\n} \\
1 & b_{11}-1 &b_{12}&b_{13}&\dots&b_{1\n}\\
\Delta_2& b_{21}+\Delta_2(b_{11}-1)  &b_{22}+\Delta_2b_{12} &b_{23}+\Delta_2b_{13}&\dots&b_{2\n}+\Delta_2b_{1\n}\\
\vdots & \vdots&\vdots&\vdots&&\vdots\\
\Delta_n& b_{n1}+\Delta_n(b_{11}-1)  &b_{n2}+\Delta_nb_{12} &b_{n3}+\Delta_nb_{13}&\dots&b_{n\n}+\Delta_nb_{1\n}
\end{pmatrix}}.$}
\]

Subtracting row 2 from row 1 and row 2 from row $j$ for all $j$ with $b_{j1} > 0$, we arrive at 
\[
\mypair{\begin{pmatrix}
d_1 + b_{11}  \\ d_1+N  + b_{11}-1\\ d_2 + b_{21} \\ \vdots \\ d_n + b_{n1}
\end{pmatrix}
}{
\begin{pmatrix}
0  &   b_{11} & b_{12}  & b_{13} & \cdots & b_{1\n} \\
1 & b_{11}-1 &b_{12}&b_{13}&\dots&b_{1\n}\\
0 & b_{21}  &b_{22} &b_{23} &\dots&b_{2\n}\\
\vdots & \vdots&\vdots&\vdots&&\vdots\\
0 & b_{n1}  &b_{n2}&b_{n3}&\dots&b_{n\n}
\end{pmatrix}}
\]
and applying  Proposition \ref{antcol} in reverse $N-1$ times to the first column, we get to 
\[
\mypair{\begin{pmatrix}
d_1 + b_{11}  \\ d_1 + b_{11} \\ d_2 + b_{21} \\ \vdots \\ d_n + b_{n1}
\end{pmatrix}
}{
\begin{pmatrix}
0  &   b_{11} & b_{12}  & b_{13} & \cdots & b_{1\n} \\
1 & b_{11}-1 &b_{12}&b_{13}&\dots&b_{1\n}\\
0 & b_{21}  &b_{22} &b_{23} &\dots&b_{2\n}\\
\vdots & \vdots&\vdots&\vdots&&\vdots\\
0 & b_{n1}  &b_{n2}&b_{n3}&\dots&b_{n\n} 
\end{pmatrix}}
\]
by moves of type \OOO, \IIIp\ and \RRRp.  

And finally, we reach
\[
\mypair{\begin{pmatrix}
d_1 + b_{11}\\  
d_2 + b_{21} \\ \vdots \\ d_n + b_{n1}
\end{pmatrix}
}{
\begin{pmatrix}
b_{11} &b_{12}&b_{13}&\dots&b_{1\n}\\
 b_{21}  &b_{22} &b_{23} &\dots&b_{2\n}\\
 \vdots&\vdots&\vdots&&\vdots\\
 b_{n1}  &b_{n2}&b_{n3}&\dots&b_{n\n} 
\end{pmatrix}}
\]
by an \OOO\ move in reverse.  
\end{proof}

\begin{propo}\label{colplus}
Given matrices $\DBpair$ describing the graph $E$ in augmented canonical form. When \vt{i} and \vt{j} are different regular vertices so that there is a path from \vt{i} to \vt{j}, then we can go from $E$ to  the graph described by the pair $\DBpairprime$ given by
\[
\DD'=\DD=\begin{pmatrix}d_1\\d_2\\d_3\\\vdots\\d_n\end{pmatrix}\quad \BB'=\begin{pmatrix}\cdots&b_{1,j-1}&b_{1j}+b_{1i}&b_{1,j+1}&\cdots\\
\cdots&b_{2,j-1}&b_{2j}+b_{2i}&b_{2,j+1}&\cdots\\
&\vdots&\vdots&\vdots&&\\\cdots&b_{n,j-1}&b_{nj}+b_{ni}&b_{n,j+1}&\cdots
\end{pmatrix}
\] 
by  moves of type \OOO, \IIIp, and \RRRp.
\end{propo}
\begin{proof}
We may assume $i=1$ and $j=2$, and because the graph is in augmented canonical form, there is an edge from \vt 1 to \vt 2. To apply Proposition \ref{col}  we  use Theorem \ref{addcolalways} four times to pass to the vector
\[
\begin{pmatrix}d_1+2(b_{11}+b_{12})\\d_2+2(b_{21}+b_{22})\\d_3+2(b_{31}+b_{32})\\\vdots\\d_n+2(b_{n1}+b_{n2})\end{pmatrix}\
\]
so that the conditions are met to get to
\[
\mypair{\begin{pmatrix}d_1+2(b_{11}+b_{12})\\d_2+2(b_{21}+b_{22})\\d_3+2(b_{31}+b_{32})\\\vdots\\d_n+2(b_{n1}+b_{n2})\end{pmatrix}}{\begin{pmatrix}b_{11}&b_{12}+b_{11}&b_{13}&\dots&b_{1\n}\\
b_{21}&b_{22}+b_{21}&b_{23}&\dots&b_{2\n}\\
b_{31}&b_{32}+b_{31}&b_{33}&\dots&b_{3\n}\\\vdots&\vdots&\vdots&&\vdots\\b_{n1}&b_{n2}+b_{n1}&b_{n3}&\dots&b_{n\n}\end{pmatrix}}.
\]
We then apply Theorem \ref{addcolalways} in reverse two times to reach the conclusion in  Proposition \ref{col} irrespective of the original $d_i$. 
\end{proof}
\section{Conclusion}

\begin{defin}
We say that two graphs $E$ and $F$ are in \emph{augmented standard form} if both are in augmented canonical form, and if there is an isomorphism $\psi:\Gamma_E\to\Gamma_F$ so that 
\[
|\psi(\gamma(i))|^\bullet=|\gamma(i)|^\bullet \qquad
|\psi(\gamma(i))|^\circ=|\gamma(i)|^\circ 
\]
for all $\gamma(i)\in \Gamma_E$.
\end{defin}

We usually identify $\Gamma=\Gamma_E=\Gamma_F$ in this case.

\begin{lemma}\label{getstd}
When $C^*(E)\simeq C^*(F)$, then we can replace $E$ by $E'$ and $F$ by $F'$ by moves of type \OOO, \IIIp, and \RRRp, so that $E'$ and $F'$ are in augmented standard form.
\end{lemma}
\begin{proof}
Place $E$ and $F$ in augmented canonical form by Proposition \ref{getcanon}. Since $\Gamma_E$ and $\Gamma_F$ are reflected in the ideal structure of the $C^*$-algebras, the $*$-isomorphism implements an order isomorphism in a way that the corresponding components define gauge simple $C^*$-algebras that are mutually isomorphic. Thus the types in the trichotomy as well as the number of singular vertices are the same. Arguing as in \textsc{Step} 4 of Proposition \ref{getcanon} we may component-wise increase the number of regular vertices on either side to match them up, and run the algorithm to the end from there to reestablish augmented canonical form.
\end{proof}

\begin{defin}
Assume that $E$ and $F$ are in augmented standard form over $\Gamma$ and set $\nn=(|\gamma|)_{\gamma \in\Gamma}$ and $\mm=(|\gamma|^\bullet)_{\gamma \in\Gamma}$. We say that the graphs are $\GL$-equivalent if there exist invertible $U\in \MG(\nn)$ and $V\in \MG(\mm)$ so that
\[
U\BB_E^\bullet=\BB_F^\bullet V.
\]
If further all diagonal blocks in $U$ and $V$ can be chosen with determinant 1, we say that the graphs are $\SL$-equivalent.

If $U$ may be chosen so that $U\DD_E-\DD_F\in \cok \BB_F^\bullet$ we say that $E$ and $F$ are $\GLp$- or $\SLp$-equivalent.
\end{defin}

Let $\myop$ the matrix which differs from the identity matrix only by a one in entry $ij$.
We have established (in a way to be made precise in the ensuing proof) that the very special $\SLp$-equivalences implemented by $(U,V)=(\myop,\mathsf I)$ or $(U,V)=(\mathsf I,\myop)$ are given by moves. This we generalize as follows:

\begin{theor} \label{SLtomoves}
Let $E,F$ be graphs in augmented standard form that  are $\SLp$-equivalent. Then 
$E$ may be transformed to $F$ by moves of the type \OOO, \IIIp\ and \RRRp.
\end{theor}

\begin{proof}
Our aim is to go from $(\DD_E,\BB_E)$ to $(\DD_F,\BB_F)$ by row and column operations, visiting graphs specified by $(\DD^{(k)},\BB^{(k)})$ in augmented canonical form along the way, with $(\DD^{(0)},\BB^{(0)})=(\DD_E,\BB_E)$ and  reaching $(\DD_F,\BB_F)$ at the end. This requires in particular that 
\[
d^{(k)}_i\geq 1\qquad b^{(k)}_{ij}\geq 0 \qquad b^{(k)}_{ii}\geq -1
\]
everywhere. 

That this is possible for the $\BB$ matrices is exactly proved in \cite[Theorem 9.10]{segrerapws:ccuggs}. More precisely, since the graphs represented by $\BB_E$ and $\BB_F$ are in standard form and $\SL$-equivalent, a sequence of ``legal'' row and column operations are specified to obtain $\BB^{(k+1)}$ from $\BB^{(k)}$, going from $\BB^{(0)}=\BB_E$ to $\BB^{(M)}=\BB_F$. It is always legal to add row $i$ to row $j$ when $\gamma(i)\geq \gamma(j)$, but to perform the corresponding row subtraction we further need to ensure that $\BB^{(k+1)}$  remains in augmented standard form with $\BB^{(k)}$ (roughly speaking by not taking too much away). The same applies to column operations, but these are further restricted to the realm of regular vertices.

\begin{figure}
\begin{center}
\begin{tabular}{r||c|c}
\cite{segrerapws:ccuggs}&Add&Subtract\\\hline\hline
Row&$\myop\BB$&
$\myop^{-1}\BB$\\\hline
Column&$\BB\myop$&$\BB\myop^{-1}$\end{tabular}
\qquad 
\begin{tabular}{r||c|c}
\emph{Ibid.}&Add&Subtract\\\hline\hline
Row&$(\myop\DD,\myop\BB)$&
$(\myop^{-1}(\BB\mathbf z+\DD),\myop^{-1}\BB)$\\\hline
Column&$(\DD,\BB\myop)$&$(\DD,\BB\myop^{-1})$\end{tabular}
\end{center}
\caption{Left: Legal  operations in \cite{segrerapws:ccuggs}. Right: Legal operations here.}\label{myops}
\end{figure}

With $\myop$ the matrix which differs by the identity matrix only by a one in entry $ij$, it is clear that the operations described are given as in the table to the left of Figure \ref{myops}. Note also that the legality conditions imply that $\myop\in \MG(\nn)$ throughout. We use the dangerous convention that $\infty-\infty=\infty$ in row subtractions, cf.\ Remark \ref{subtract}.

The operation matrices implement $U$ and $V$ in the sense that the product of matrices acting from the left is $U$ and the product of the matrices acting from the right becomes $V$ after  deletion of row and columns corresponding to singular vertices. We denote by $U^{(k)}$ the matrix obtained by multiplying all operator matrices applied from the left to reach step $k$, with $U^{(k+1)}=U^{(k)}$ whenever the operation is performed on the right. 
 
 It follows directly from Propositions \ref{rowplus} and \ref{colplus} that we can extend three of these operations to pairs representing graphs in augmented canonical form, but row subtraction requires care, since we may only meaningfully subtract row $i$ from row $j$ when $d_j>d_i$. However, as a consequence of the fact that  every matrix $\BB^{(k)}$ is in augmented standard form with its predecessor, one may check that the procedure given in \cite[Section 9]{segrerapws:ccuggs} never makes a row subtraction of row $i$ from row $j$ unless there is a regular $\vt \ell$ so that $b_{i\ell}<b_{j\ell}$. Hence we may use Theorem \ref{addcolalways} to pass to a pair where $d_i<d_j$ before effectuating the operation. This is indicated to the right of  Figure \ref{myops}. Here $\mathbf z$ is a multiple of the basis vector $\mathbf e_\ell$; in particular it vanishes on all $i$ corresponding to singular \vt i.

We define $(\DD^{(k)},\BB^{(k)})$ by these operations, and claim that 
\begin{equation}\label{kicker}
\DD^{(k)}=U^{(k)} \DD^{(0)}+\BB^{(k)}\mathbf x^{(k)}
\end{equation}
with $\mathbf x^{(k)}$ a vector which vanishes on all singular entries. Indeed for row subtractions we have
\[
\DD^{(k+1)}=\myop^{-1}(\BB^{(k)}\mathbf z+\DD^{(k)})=\BB^{(k+1)}\mathbf z+\myop^{-1}U^{(k)} \DD^{(0)}+\myop^{-1}\BB^{(k)}\mathbf x^{(k)}=
U^{(k+1)} \DD^{(0)}+\BB^{(k+1)}\mathbf x^{(k+1)}
\]
with $ \mathbf x^{(k+1)}=\mathbf z+\mathbf x^{(k)}$, and the same with $\mathbf z=\OO$ for row additions. For column operations we have $\DD^{(k+1)}=\DD^{(k)}$, but we must set $\mathbf x^{(k+1)}$ to either $\myop^{-1} \mathbf x^{(k)}$ or $\myop \mathbf x^{(k)}$ as appropriate.

We now know  that we can go from $(\DD_E,\BB_E)=(\DD^{(0)},\BB^{(0)})$ to the pair $(\DD^{(M)},\BB^{(M)})$ using moves \OOO,\IIIp, and \RRRp, and we know that $\BB^{(M)}=\BB_F$ as desired. We have by \eqref{kicker} that $\DD^{(M)}$ and $U\DD^{(0)}$ define the same element in $\cok \BB^{(M)}$, and by  our assumption we know this is also the same element as the one defined by $\DD_F$. In other words, we can write
\[
\DD^{(M)}-\DD_F=\BB_F\mathbf y
\]
with $y_i=0$ for singular $i$. Redistributing according to signs we get
\[
\DD^{(M)}+\sum_{i=1}^\n{y_i'}(\BB_F)_i=\DD_F+\sum_{i=1}^\n{y_i''}(\BB_F)_i=\DD''
\]
with all $y_i',y_i''\geq 0$, so we may apply Theorem \ref{addcolalways} to take both $(\DD^{(M)},\BB_F)$ and $(\DD_F,\BB_F)$ to $(\DD'',\BB_F)$. 

Finally, we note that after reorganizing the vertices in each component so that the singuar vertices are listed last, the columns in $\BB_E$ and $\BB_F$ corresponding to singular vertices are in fact identical, since they are completely determined by the information in $\Gamma$ because of (III) in the definition of augmented canonical form. These columns will not be affected by the moves we   performed, and hence this part of the matrices require no further attention. The proof is complete.

\end{proof}

\begin{corol}\label{AER-cor:main-5}
Let $E$ and $F$ be graphs with finitely many vertices.  
Then the following are equivalent
\begin{enumerate}[(1)]
\item \label{AER-cor:main-5-item-1}  $E$ can be obtained by $F$ by moves of the type \OOO,\IIIp,\RRRp,\CCCp,\PPPp,
\item\label{AER-cor:main-5-item-2} 
$C^{*} (E )\cong C^{*} (F)$, and,
\item\label{AER-cor:main-5-item-3} 
The filtered, ordered, pointed  $K$-theories of $C^*(E)$ and $C^*(F)$ are isomorphic.
\end{enumerate}
When $E$ and $F$ are in augmented standard form, they are also equivalent to
\begin{enumerate}[(1)]\addtocounter{enumi}{3}
\item \label{AER-cor:main-5-item-4}  $E$ is $\GLp$-equivalent to $F$.
\end{enumerate}
\end{corol}

\begin{proof}
We proved $(2)\Longleftrightarrow (3)$ in \cite{segrerapws:ccuggs}, and $(1)\Longrightarrow (2)$ was proved in \cite{seer:rmsigc} as noted in Theorem \ref{fromcompanion}. 

Assuming (2), we note by Lemma \ref{getstd} that we may pass without loss of generality to the case when $E$ and $F$ are in augmented standard form. It is proved in \cite[Theorem 14.6]{segrerapws:ccuggs} that (4) then holds. Appealing further to  \cite[Section 11-12]{segrerapws:ccuggs}, we may change the graphs by moves to arrive at two graphs that are $\SLp$-equivalent. Indeed, in these two sections a pair of graphs in standard form are revised by a finite number of changes of the form
\begin{itemize}
\item The move \CCC,
\item The move \PPP,
\item Simple expansions by move \RRR\ in reverse.
\end{itemize} 
to arrange for the orginal pair $(U,V)$ giving a $\GL$-equivalence to be replaced by one with determinants 1 by an inductive procedure. The procedure also involves
rearranging for standard form by a number of row operations after each step.

Starting with a pair of graphs in augmented standard form being $\GLp$-equivalent, we do the same, but use \CCCp, \PPPp, and \RRRp\ instead to obtain a pair that is $\SLp$-equivalent. To do the \RRRp\ move in reverse we have to have an antenna to delete, but since these changes are only applied to vertices $i$ for which $\gamma(i)$ satisfies (IV) of augmented canonical form, this is easily arranged by Theorem \ref{addcolalways}. Applying the algorithm in Proposition \ref{getcanon} from \textsc{Step} 4 onwards reestablishes augmented standard form without changing $\SLp$-equivalence.

The argument is completed by Theorem \ref{SLtomoves}.
\end{proof}

\begin{remar}
We are deliberately vague about the $K$-theoretical invariant from (2) above -- see \cite{segrerapws:ccuggs} for details. As explained there, all conditions are decidable because of \cite{bs:decide}.
\end{remar}

\end{document}